\documentclass[b5paper,twoside,headrule,reqno,11pt,]{amsart}
\def\thefootnote{}
\footskip=10mm\parskip 0.2cm\addtocounter{page}{0}
scaled \magstep0

\usepackage{amssymb}
\usepackage{url}
\newcommand{\Q}{{\mathbb{Q}}}
\newcommand{\F}{{\mathbb{F}}}
\newcommand{\C}{{\mathbb{C}}}
\newcommand{\Z}{{\mathbb{Z}}}

\newcommand{\K}{{\mathbb{K}}}

\newcommand{\bB}{{\boldsymbol{B}}}
\newcommand{\bG}{{\boldsymbol{G}}}
\newcommand{\bH}{{\boldsymbol{H}}}
\newcommand{\bL}{{\boldsymbol{L}}}
\newcommand{\bT}{{\boldsymbol{T}}}
\newcommand{\bU}{{\boldsymbol{U}}}
\newcommand{\bW}{{\boldsymbol{W}}}

\newcommand{\fS}{{\mathfrak{S}}}

\newcommand{\cC}{{\mathcal{C}}}
\newcommand{\cE}{{\mathcal{E}}}

\newcommand{\cH}{{\mathcal{H}}}

\newcommand{\cN}{{\mathcal{N}}}
\newcommand{\cO}{{\mathcal{O}}}

\newcommand{\ii}{{\textrm{i}}}
\newcommand\Irr{\operatorname{Irr}}
\newcommand\Unip{\operatorname{Unip}}
\newcommand\Ind{\operatorname{Ind}}

\newcommand{\RTG}{{R_\bT^\bG}}

\newtheorem{thm}{Theorem}[section]
\newtheorem{cor}[thm]{Corollary}
\newtheorem{prop}[thm]{Proposition}
\newtheorem{lem}[thm]{Lemma}

\theoremstyle{definition}

\newtheorem{exmp}[thm]{Example}
\newtheorem{abs}[thm]{}

\theoremstyle{remark}
\newtheorem{rem}[thm]{Remark}

\numberwithin{equation}{section}
\numberwithin{table}{section}

\renewcommand{\leq}{\leqslant}
\renewcommand{\geq}{\geqslant}

\begin{document}
%%================================head--document=======================

\def\evenhead{{\protect\centerline{\textsl{\large{Meinolf Geck}}}\hfill}}

\def\oddhead{{\protect\centerline{\textsl{\large{On the computation
of character values ...}}}\hfill}}

\pagestyle{myheadings} \markboth{\evenhead}{\oddhead}

\thispagestyle{empty} %\noindent{{\small\rm Pure and Applied
%Mathematics Quarterly\\ Volume 7, Number 3\\ (\textit{Special Issue:
%In honor  of \\Professor Jacques Tits})\\
%587---620, 2011}} \vspace*{1.5cm} \normalsize

\begin{center}
\Large {\bf On the computation of character values for finite 
Chevalley groups of exceptional type}
\end{center}
%\footnotetext{Received June 9, 2008.}

\begin{center}
\large Meinolf Geck

%\textit{ \small Dedicated to ??}
\textit{ \small }

\end{center}

\footnotetext{ \emph{2000 Mathematics Subject
Classfication.}{Primary 20C33; Secondary 20G40.}}

\renewcommand{\thefootnote}{\fnsymbol{footnote}}

\begin{center}
\begin{minipage}{5in}
\noindent {\bf Abstract:} We discuss various computational issues 
around the problem of determining the character values of finite
Chevalley groups, in the framework provided by Lusztig's theory of
character sheaves. Some of the remaining open questions (concerning 
certain roots of unity) for the cuspidal unipotent character sheaves 
of groups of exceptional type are resolved.
\\
\noindent{\bf Keywords:} Groups of Lie type, Deligne--Lusztig characters,
character sheaves. 
\end{minipage}
\end{center}

%%%%%%%%%%%%%%%%%%%%%%%%%%%%%%%%%%%%%%%%%%%%%%%%%%%%%%%%%%%%%%%%%%%%%%%%%%%
\section{Introduction} \label{sec0}

Let $p$ be a prime and $k=\overline{\F}_p$ be an algebraic closure of
the field with $p$ elements. Let $\bG$ be a connected reductive algebraic
group over $k$ and assume that $\bG$ is defined over the finite subfield
$\F_q\subseteq k$, where $q$ is a power of~$p$. Let $F\colon \bG\rightarrow 
\bG$ be the corresponding Frobenius map. Then the group of rational points 
$\bG^F=\bG(\F_q)$ is called a ``finite group of Lie type''. (For the basic 
theory of these groups, see \cite{Ca2}, \cite{rDiMi}, \cite{gema}.) We are 
concerned with the problem of computing the values of the irreducible 
characters of $\bG^F$. The work of Lusztig \cite{LuB}, \cite{Lintr}, 
\cite{L7}, \cite{Lu18} has led to a general program for solving this 
problem. In this framework, one seeks to establish certain identities of 
class functions on $\bG^F$ of the form $R_x=\zeta\chi_A$, where $R_x$ 
denotes an ``almost character'' (that is, an explicitly defined linear 
combination of the irreducible characters of $\bG^F$) and $\chi_A$ denotes 
the characteristic function of a suitable $F$-invariant ``character sheaf'' 
$A$ on $\bG$; here, $\zeta$ is an algebraic number of absolute value~$1$. 
This program has been successfully carried out in many cases (see 
\cite[\S 2.7]{gema} for a survey), but not in complete generality. 

This paper is part of an ongoing project (involving various authors) 
to complete the program of establishing identities $R_x=\zeta\chi_A$ 
as above including the explicit determination of the scalars $\zeta$. 
We shall solve this problem here in a number of previously open cases for 
$\bG$ simple of exceptional type. 

The above identities $R_x=\zeta\chi_A$ take a particularly striking form 
when $A$ is a cuspidal character sheaf and $\bG$ is a simple algebraic 
group, and this is our main focus here. In that case, the set $\{g\in
\bG^F\mid \chi_A(g)\neq 0\}$ is contained in a single $F$-stable 
conjugacy class $\Sigma$ of~$\bG$; furthermore, the values of $\chi_A$ are
determined by the choice of an element $g_1\in\Sigma^F$ and a certain 
irreducible character $\psi$ of the finite group of components
$A_\bG(g_1)=C_\bG(g_1)/C_\bG^\circ(g_1)$. By \cite[0.4]{L7}, the general 
case can be reduced to the ``cuspidal'' case assuming that the 
cohomological induction functor~$R_\bL^\bG$ (see \cite{Lu1a}, 
\cite[\S 9.1]{rDiMi}) is explicitly known, where $\bL\subseteq \bG$ is any 
$F$-stable Levi subgroup of a not necessarily $F$-stable parabolic 
subgroup of~$\bG$.

A crucial ingredient in this whole program is the problem of identifying 
``good'' choices for $g_1\in \Sigma^F$ as above. If $\Sigma$ is a 
unipotent class, then one can use the concept of ``split'' elements; see 
Beynon--Spaltenstein \cite{BeSp} and Shoji \cite{S6}. In general, there 
are a few rare cases where one can single out a representative $g_1\in
\Sigma^F$ simply by looking at the order or the structure of the 
centralisers. At the other extreme, all $g_1\in \Sigma^F$ may have the 
same centraliser order. In such cases, we use the following techniques:
1)~Steinberg's cross-section \cite{St} for regular elements or, more 
generally, Lusztig's ``$C$-small'' classes \cite{Lfrom}; 2)~rationality 
properties of characters; 3)~powermaps and congruence conditions for 
character values. 

Our aim is to achieve something close to the famous Cambridge 
{\sf ATLAS}~\cite{atl}, where no explicit representatives of conjugacy 
classes are given, but the classes can be almost uniquely identified by
some formal properties, like~2) or~3). For many applications of character
theory (for examples, see \cite[App.~A.10]{gema}) this is entirely 
sufficient. In Section~\ref{sec3}, we develop some techniques that will 
help us identifying such ``good'' choices for $g_1\in\Sigma^F$. (Note, 
however, that there does not yet seem to be a universally applicable 
definition of what ``good'' should mean.)

But first we need to address another essential point: the explicit
evaluation of the Deligne--Lusztig characters $R_\bT^\bG(\theta)$, 
where $\bT$ is any $F$-stable maximal torus of~$\bG$ and $\theta
\in\Irr(\bT^F)$. There is a character formula in \cite{DeLu} which 
reduces that problem to the computation of Green functions. The formula 
involves some technical issues of a purely group-theoretical nature. It 
will be known to the experts how to deal with this, but the details are 
not readily available so we include them here in Sections~\ref{sec1}
and \ref{sec2} (following, and slightly refining L\"ubeck~\cite{lphd}). 
We hope that this will be useful as a reference in other contexts 
as well.

Finally, in Sections~\ref{sece6}--\ref{secf4}, we explicitly deal with 
cuspidal character sheaves in groups of types $F_4$, $E_6$, $E_7$. Much of 
this is inspired by Lusztig \cite{L3} (values of characters on unipotent 
elements) and Shoji \cite{Sclass} (values of unipotent characters for 
classical groups); an additional complication here is that unipotent
characters of exceptional groups may have non-rational values. We heavily 
rely on computer calculations, where we use Michel's 
extremely powerful version of {\sf CHEVIE} \cite{jmich}. In addition to
the general functions concerning Weyl groups, reflection subgroups and
their character tables, there are programs in \cite{jmich} for producing 
information about the unipotent characters of $\bG^F$ (degrees, Fourier 
matrices etc.), and for computing (generalised) Green functions, which 
turn out to be particularly helpful for our purposes here. Combined with 
previous work by a number of authors (for precise references see 
Sections~\ref{sece6}--\ref{secf4}), we can now state:

\textit{Let $\bG$ be simple of type $G_2$, $F_4$, $E_6$ or $E_7$. Then 
the scalars $\zeta$ in the identities $R_x=\zeta \chi_A$ for cuspidal 
unipotent character sheaves $A$ are explicitly known. In all cases 
considered, there is a ``good'' choice of $g_1\in\Sigma^F$ such that 
$\zeta=1$.}

As far as simple groups of exceptional type are concerned, what remains to 
be done is to deal with a number of cuspidal character sheaves for $\bG$
of type $E_8$ (which are all unipotent) and with the non-unipotent cuspidal 
character sheaves for $\bG$ of type $E_6$ and $E_7$. We hope to address
these questions elsewhere.

\smallskip
\begin{center} {\footnotesize
\begin{tabular}{lp{270pt}r} \multicolumn{3}{c}{\bf Table of Contents} \\
1 & Introduction \dotfill & \pageref{sec0}\\ 2 & Fusion of $F$-stable
maximal tori \dotfill & \pageref{sec1}\\ 3 & On the evaluation of 
Deligne--Lusztig characters \dotfill & \pageref{sec2}\\ 4 & Characteristic 
functions and conjugacy classes \dotfill & \pageref{sec3}\\ 5 & Cuspidal
unipotent character sheaves in type $E_6$ \dotfill & \pageref{sece6}\\
6 & Cuspidal unipotent character sheaves in type $E_7$ \dotfill & 
\pageref{sece7}\\ 7 & Cuspidal character sheaves in type $F_4$ \dotfill 
& \pageref{secf4}
\end{tabular}}
\end{center}

\smallskip
\begin{abs} {\bf Notation and conventions.} \label{sub00}
The set of (complex) irreducible characters of a finite group~$\Gamma$
is denoted by $\Irr(\Gamma)$. We work over a fixed subfield $\K\subseteq 
\C$, which is algebraic over $\Q$, invariant under complex conjugation 
and ``large enough'', that is, $\K$ contains sufficiently many roots of 
unity and $\K$ is a splitting field for $\Gamma$ and all its subgroups. 
Thus, $\chi(g)\in\K$ for all $\chi\in\Irr(\Gamma)$ and $g\in\Gamma$. 
When required, we will assume chosen an embedding of $\K$ into 
$\overline{\Q}_l$, where $\Q_l$ is the field of $l$-adic numbers for some 
prime~$l$. If $\alpha\colon \Gamma\rightarrow\Gamma$ is a group 
automorphism, we say that $g_1,g_2\in\Gamma$ are $\alpha$-conjugate if
there exists some $g\in\Gamma$ such that $g_2=g^{-1}g_1\alpha(g)$.
\end{abs}

%%%%%%%%%%%%%%%%%%%%%%%%%%%%%%%%%%%%%%%%%%%%%%%%%%%%%%%%%%%%%%%%%%%%%%%%%%%
\section{Fusion of $F$-stable maximal tori} \label{sec1}

We keep the general notation from the introduction.
Let $\bT_0$ be a maximally split torus of $\bG$, that is, $\bT_0$
is an $F$-stable maximal torus contained in an $F$-stable Borel 
subgroup $\bB\subseteq \bG$. Let $\Phi$ be the root system of $\bG$
with respect to $\bT_0$, and let $\Phi^+\subseteq \Phi$ be the set of 
positive roots determined by~$\bB$. Let $\bW:=N_\bG(\bT_0)/\bT_0$ be 
the Weyl group of $\bT_0$ and $\ell\colon \bW\rightarrow\Z_{\geq 0}$ be the
length function. We have $\bW=\langle w_\alpha \mid \alpha\in
\Phi\rangle$ where $w_\alpha\in\bW$ denotes the reflection with root 
$\alpha$. We have $\bG=\langle \bT_0,\bU_\alpha\;(\alpha\in\Phi)\rangle$
where $\bU_\alpha\subseteq \bG$ denotes the root subgroup corresponding to
$\alpha$. The Frobenius map $F$ induces a permutation $\alpha\mapsto
\alpha^\dagger$ of $\Phi$ such that $F(\bU_\alpha)=\bU_{\alpha^\dagger}$
for all $\alpha\in\Phi$. We denote by $\sigma\colon \bW\rightarrow\bW$ the 
automorphism induced by~$F$. For each $w\in \bW$, let $\dot{w} \in
N_\bG(\bT_0)$ be a representative; if $\sigma(w)=w$, then we tacitly 
assume that $F(\dot{w})=\dot{w}$. 

It is well known that the $\bG^F$-conjugacy classes of $F$-stable maximal 
tori of $\bG$ are parametrised by the $\sigma$-conjugacy classes of $\bW$. 
Given $w\in \bW$, let $g\in \bG$ be such that $\dot{w}=g^{-1}F(g)$. (The
existence of~$g$ relies on Lang's Theorem, which will be used many
times in what follows, without further explicit reference.) Then $\bT:=
g\bT_0g^{-1}$ is a corresponding $F$-stable maximal torus, unique up to 
conjugation by elements of $\bG^F$; in this situation, we also say that 
$\bT$ is ``of type~$w$''.  (See \cite[\S 2.3]{Ca2} for further details.) 
For the evaluation of Deligne--Lusztig characters, we shall need to 
relate $\bG^F$-conjugacy classes of $F$-stable maximal tori of $\bG$ to 
those in certain connected reductive subgroups of maximal rank. Since
this is crucial for the explicit computations that we need to carry out, 
we will explain the details here; see also L\"ubeck \cite{lphd}, 
\cite{Lue07}.

\begin{abs} {\bf Subsystem subgroups.} As in Carter \cite[\S 2]{C1}, we 
consider subsets $\Phi' \subseteq \Phi$ that are themselves root systems 
and are closed in the sense that, whenever $\alpha,\beta\in\Phi'$ and 
$\alpha+\beta\in \Phi$, then $\alpha+\beta\in \Phi'$. Given such a $\Phi'$,
there is a corresponding closed connected reductive subgroup $\bH'\subseteq
\bG$ generated by $\bT_0$ and the subgroups $\bU_\alpha$ for $\alpha \in 
\Phi'$. Here, $\Phi'$ is the root system of $\bH'$ with respect to $\bT_0$ 
and 
\[ \bW(\Phi'):=\langle w_\alpha\mid \alpha\in\Phi'\rangle=(N_\bG(\bT_0)\cap
\bH')/\bT_0\]
is the Weyl group of $\bH'$. Let $\Xi$ be the set of all pairs $(\Phi',w)$ 
where $\Phi'\subseteq \Phi$ is a subset as above and $w\in \bW$ is such that
$w(\alpha^\dagger)\in \Phi'$ for all $\alpha\in \Phi'$. Given $(\Phi',w)
\in~\Xi$, we form the corresponding subgroup $\bH'\subseteq \bG$ as above;
then $F(\bH')=\dot{w}^{-1}\bH'\dot{w}$; note that $\dot{w}\bU_\alpha
\dot{w}^{-1}=\bU_{w(\alpha)}$ for all $\alpha\in\Phi$. Hence, writing 
$\dot{w}=g^{-1}F(g)$ for some $g\in \bG$, the subgroup $g\bH'g^{-1}$ is 
$F$-stable and uniquely determined by $(\Phi',w)$, up to conjugation by 
an element of $\bG^F$. It is known (see Carter \cite{C1}, Deriziotis 
\cite{Dere}, Mizuno \cite{Miz}) that the $\bG^F$-conjugacy classes
of $F$-stable subgroups $g\bH'g^{-1}$ as above are parametrised by the 
pairs in $\Xi$ modulo the equivalence relation defined by: $(\Phi_1', 
w_1) \sim (\Phi_2',w_2)$ if there exists some $x\in\bW$ such that 
$x(\Phi_1')=\Phi_2'$ and $x^{-1}w_2\sigma(x)w_1^{-1}\in\bW(\Phi_1')$. 
(The above statement concerning $\bG^F$-conjugacy classes of $F$-stable 
maximal tori is a very special case of this correspondence.)
\end{abs}

\begin{abs} {\bf The relation $\sim$ on $\Xi$.} \label{rem0} For future 
reference, we briefly indicate how the relation $\sim$ comes about. (Note 
that the discussion in \cite[\S 2]{C1} assumes that the subgroup $\bH'$ 
corresponding to $\Phi'$ is itself $F$-stable, which will not always be 
the case if $\sigma\neq \mbox{id}_\bW$; furthermore, \cite[\S 2]{C1} only 
considers $\bG^F$-conjugacy for the subgroups corresponding to a fixed 
$\Phi'$.) So let $(\Phi_1',w_1)$ and $(\Phi_2',w_2)$ be pairs in $\Xi$; 
let $g_1,g_2\in\bG$ be such that $g_1^{-1}F(g_1)=\dot{w}_1$ and $g_2^{-1}
F(g_2)=\dot{w}_2$. We have the corresponding $F$-stable subgroups 
$g_i\bH_i'g_i^{-1}$, where $\bH_i':=\langle \bT_0,\bU_\alpha\; (\alpha
\in \Phi_i')\rangle$ for $i=1,2$. Suppose now that $g_1\bH_1'
g_1^{-1}$ and $g_2\bH_2'g^{-1}$ are conjugate in $\bG^F$; so 
$\tilde{g}g_1\bH_1'g_1^{-1} \tilde{g}^{-1}=g_2\bH_2'g_2^{-1}$ for
some $\tilde{g}\in \bG^F$. Setting $\hat{g}:=g_2^{-1}\tilde{g}g_1$, we 
have $\hat{g}\bH_1'\hat{g}^{-1}=\bH_2'$ and there exists some 
$h_2\in\bH_2'$ such that $\hat{g}\bT_0\hat{g}^{-1}=h_2\bT_0h_2^{-1}$. 
Then $n:=h_2^{-1}\hat{g}\in N_\bG(\bT_0)$ and $n\bH_1'n^{-1}=\bH_2'$. 
Hence, $x(\Phi_1')=\Phi_2'$, where $x$ is the image of $n$ in~$\bW$. 
We now have $g_1=\tilde{g}^{-1}g_2\hat{g}=\tilde{g}^{-1}g_2h_2n$ and 
a straightforward computation yields that 
\[\dot{w}_1=g_1^{-1}F(g_1)=\bigl(n^{-1}\dot{w}_2F(n)\bigr)\bigl(F(n)^{-1}
\dot{w}_2^{-1} h_2^{-1}\dot{w}_2F(h_2)F(n)\bigr).\]
We have $F(\bH_2')=\dot{w}_2^{-1}\bH_2'\dot{w}_2$ and so 
$\dot{w}_2^{-1}h_2^{-1}\dot{w}_2F(h_2)\in F(\bH_2')$.
Furthermore, $F(n)^{-1}F(\bH_2')F(n)=F(n^{-1}\bH_2'n)=F(\bH_1')=
\dot{w}_1^{-1}\bH_1'\dot{w}_1$. Hence, we obtain $\dot{w}_1=
n^{-1}\dot{w}_2F(n)\dot{w}_1^{-1}h_1\dot{w}_1$ for some $h_1\in \bH_1'$.
Thus, $n^{-1}\dot{w}_2F(n)\dot{w}_1^{-1}\in N_G(\bT_0)\cap\bH_1'$ 
and so $x^{-1}w_2\sigma(x)w_1^{-1}\in \bW(\Phi_1')$, as desired. 
Conversely, if $(\Phi_1',w_1) \sim (\Phi_2',w_2)$, then one needs to
run the above argument backwards.
\end{abs}

\begin{abs} \label{rem1} 
Let us fix a pair $(\Phi',w)\in\Xi$. Note that $(\Phi',w)\sim (\Phi',
uw)$ for all $u\in \bW(\Phi')$. Now, by \cite[Lemma~1.9]{LuB}, the coset
$\bW(\Phi')w$ contains a unique element of minimal length; let us denote 
this element by~$d$. Thus, when considering equivalence classes of pairs 
$(\Phi',d)\in \Xi$, we may assume without loss of generality that $d$ 
has minimal length in the coset $\bW(\Phi')d$. (Note that L\"ubeck 
\cite{lphd} does not make this assumption on~$d$.) We define a new 
Frobenius map $F'\colon \bG \rightarrow\bG$ by $F'(g):=\dot{d}F(g)
\dot{d}^{-1}$ for $g\in\bG$. Then $F'(\bH')= \bH'$, where $\bH'=\langle 
\bT_0,\bU_\alpha\; (\alpha\in\Phi') \rangle$. The map induced by $F'$
on $\bW$ is given by 
\begin{equation*}
\sigma'\colon \bW\rightarrow \bW,\qquad w\mapsto d\sigma(w)d^{-1}.\tag{a}
\end{equation*}
Clearly, $\bT_0$ is also an $F'$-stable maximal torus of $\bH'$. We claim
that
\begin{equation*}
\mbox{$\bT_0$ is a maximally split torus of $\bH'$ with respect to $F'$}.
\tag{b}
\end{equation*}
This is seen as follows. The group $\bB'=\langle \bT_0,\bU_\alpha\;
(\alpha\in\Phi^+\cap \Phi')\rangle$ is a Borel subgroup of $\bH'$ (see 
\cite[\S 3.5]{Ca2}). Since $\bT_0\subseteq \bB'$, it is sufficient to
show that $\bB'$ is $F'$-stable. For this purpose, let $\alpha\in
\Phi^+\cap \Phi'$. By \cite[Lemma~1.9]{LuB}, we have $d^{-1}(\alpha)\in
\Phi^+$ and so $\dot{d}^{-1}\bU_\alpha\dot{d}=\bU_{d^{-1}(\alpha)}\subseteq 
\bB=F(\bB)$. Consequently, we have $\bU_\alpha\subseteq \dot{d}
F(\bB) \dot{d}^{-1}=F'(\bB)$. Since also $\bT_0=F'(\bT_0)\subseteq 
F'(\bB)$, we conclude that $\bB'\subseteq F'(\bB)$. Furthermore,
$\bB'\subseteq \bH'=F'(\bH')$ and so $\bB'\subseteq F'(\bB) \cap 
F'(\bH')=F'(\bB\cap \bH')=F'(\bB')$. Hence, we must have $\bB'=
F'(\bB')$, as claimed. Thus, if $\Delta'$ is the unique set of simple
roots in $\Phi^+\cap\Phi'$, then we have $\bW(\Phi')=\langle S'\rangle$ 
where 
\begin{equation*}
S':=\{w_\alpha\mid \alpha\in\Delta'\} \qquad\mbox{and}\qquad \sigma'(S')=S'.
\tag{c}
\end{equation*}
In particular, $(\bW(\Phi'),S')$ is a Coxeter system and $\sigma'\bigl(
\bW(\Phi')\bigr)=\bW(\Phi')$. 
\end{abs}

\begin{abs} \label{rem2} In the setting of \S \ref{rem1}, where
$(\Phi,d)\in\Xi$, let us also fix an element $g\in \bG$ such that $g^{-1}
F(g)=\dot{d}$. Then $\bT_d:=g\bT_0g^{-1}\subseteq \bG$ is an $F$-stable 
maximal torus of type~$d$. Furthermore, if $\bH'=\langle \bT_0,
\bU_\alpha\;(\alpha\in\Phi')\rangle$ and $\bH_d:=g\bH'g^{-1}$, then
$\bT_d\subseteq \bH_d$ and $F(\bH_d)=\bH_d$. Now Remark~\ref{rem1}(b)
immediately implies that $\bB_d:=g\bB'g^{-1}$ is an $F$-stable Borel 
subgroup of $\bH_d$ and so $\bT_d$ is a maximally split torus of $\bH_d$.
Let $\bW_d:=N_{\bH_d}(\bT_d)/\bT_d$ be the Weyl group of $\bH_d$. We 
denote by $\sigma_d\colon \bW_d\rightarrow \bW_d$ the automorphism 
induced by~$F$. Then the conjugation map $\gamma_g \colon 
\bG\rightarrow \bG$, $x \mapsto g^{-1}xg$, induces an embedding
$\bar{\gamma}_g\colon \bW_d \hookrightarrow \bW$, where 
\[ \bW(\Phi')=\bar{\gamma}_g(\bW_d)\subseteq \bW \qquad\mbox{and}
\qquad \bar{\gamma}_g\circ \sigma_d=\sigma' \circ \bar{\gamma}_g.\] 
Via the isomorphism $\bar{\gamma}_g\colon \bW_d\rightarrow\bW(\Phi')$, 
the $\sigma_d$-conjugacy classes of $\bW_d$ correspond to the 
$\sigma'$-conjugacy classes of $\bW(\Phi')$. Thus, the $\bH_d^F$-conjugacy 
classes of $F$-stable maximal tori of $\bH_d$ are parametrised by the 
$\sigma'$-conjugacy classes of $\bW(\Phi')$. More precisely, if $w'\in 
\bW(\Phi')$, then an $F$-stable maximal torus $\bT' \subseteq \bH_d$ 
of type~$w'$ (inside $\bH_d$) is given by $\bT':=h\bT_dh^{-1}$ where 
$h\in \bH_d$ is such that $h^{-1}F(h)=\gamma_g^{-1}(\dot{w}')=
g\dot{w}'g^{-1}$. 
\end{abs}

The following result describes the fusion from $F$-stable maximal tori
in $\bH_d$ to $F$-stable maximal tori in $\bG$; see also 
L\"ubeck \cite[\S 4.1(2)]{lphd} (but note that slightly different
conventions and assumptions are used in \cite{lphd}). 

\begin{lem} \label{lem1} In the above setting, let $\bT'\subseteq \bH_d$
be an $F$-stable maximal torus of type $w'\in\bW(\Phi')$. Then $\bT'
\subseteq \bG$ is an $F$-stable maximal torus of type $w'd\in\bW$.
In particular, a maximally split torus of $\bH_d$ is of type~$d$
(relative to $\bG$).
\end{lem}

\begin{proof} Recall that $g^{-1}F(g)=\dot{d}$ and that $\bT_d=g\bT_0 
g^{-1}$ is a maximally split torus of $\bH_d$. As above, let $h\in\bH_d$ 
be such that $\bT'=h\bT_dh^{-1}$ and $h^{-1}F(h)=\gamma_g^{-1}(\dot{w}')=
g\dot{w}' g^{-1}$. Then $\bT'=hg\bT_0g^{-1} h^{-1}$ and 
$(hg)^{-1} F(hg) \in N_\bG(\bT_0)$. Now 
\[(hg)^{-1}F(hg)=g^{-1}h^{-1}F(h)F(g)=g^{-1}(g\dot{w}'g^{-1})F(g)=
\dot{w}'\dot{d}.\]
Hence, $\bT'$ is an $F$-stable maximal torus of type $\dot{w}d$ in $\bG$.
\end{proof}

\begin{exmp} \label{exg2a} Let $\bG$ be simple of type $G_2$; then 
$\sigma=\mbox{id}_\bW$ and the permutation $\alpha\mapsto \alpha^\dagger$
is the identity. Let $\Delta=\{\alpha_1,\alpha_2\}$ be the set of simple 
roots in $\Phi^+$, where $\alpha_1$ is long and $\alpha_2$ is short. 
There are two particular subsystems $\Phi'\subseteq \Phi$ that occur in 
the classification of cuspidal character sheaves on $\bG$ (see the 
proof of \cite[Prop.~20.6]{L2d}). Up to $\bW$-conjugacy, these are 
$\Phi_0'$ of type $A_1\times A_1$, spanned by $\{\alpha_2,\alpha_0\}$, 
and $\Phi_0'' \subseteq \Phi$ of type $A_2$, spanned by $\{\alpha_1,
\alpha_0\}$. (Here, $\alpha_0\in\Phi$ denotes the unique root of maximal 
height in $\Phi$.) There is only one equivalence class of pairs
$(\Phi',w)\in\Xi$ under~$\sim$ where $\Phi'=\Phi_0'$; a representative
is given by $(\Phi_0',d_1)$ with $d_1=1_\bW$. There are two 
equivalence classes of pairs $(\Phi',w)\in\Xi$ where $\Phi'=\Phi_0''$; 
representatives are given by $(\Phi_0'',d_1)$ with $d_1=1_\bW$, and
by $(\Phi_0'',d_2)$ with $d_2=w_{\alpha_2}$ (and $d_2$ has minimal 
length in $\bW(\Phi_0'')d_2$). The information is summarised in the
following table (which is a model for the tables in later sections).
\[\begin{array}{c@{\hspace{15pt}}c@{\hspace{15pt}}
l@{\hspace{15pt}}c@{\hspace{15pt}}c} \hline \Phi' & 
\Delta' & d_i & \mbox{permutation} & \mbox{$\sigma_i'$-classes} \\ \hline 
A_1{\times} A_1 & \alpha_2,2\alpha_1{+}3\alpha_2&  d_1=1_\bW & () & 4
\\\hline A_2 & \alpha_1,\alpha_1{+}3\alpha_2&  d_1=1_\bW & () 
& 3 \\ & & d_2=w_{\alpha_2} &  (1,2) & 3 \\\hline
\end{array}\]
Here, $\Delta'$ is the set of simple roots in $\Phi^+\cap\Phi'$. 
Furthermore, we define $\sigma_i'\in \mbox{Aut}(\bW(\Phi'))$ by $\sigma_i'
(w)= d_iwd_i^{-1}$ for $w\in\bW$. Then $\sigma_i'$ induces a permutation 
of the simple reflections in $\bW(\Phi')$; see Remark~\ref{rem1}. This 
permutation, in cycle notation, is indicated in the fourth column of 
the table; note that this permutation refers to the simple roots
in~$\Delta'$, not to those in $\Delta$. The last column contains the
number of $\sigma_i'$-conjugacy classes of $\bW(\Phi')$.

Now consider the fusion of $F$-stable maximal tori described by
Lemma~\ref{lem1}. In each case, we need to work out representatives 
of the $\sigma_i'$-conjugacy classses of $\bW(\Phi')$, multiply these 
by $d_i$ and identify the conjugacy class of $\bW$ to which the new 
element belongs. Here, of course, this can be done by hand, but for 
larger~$\bW$, such computations are conveniently done using the computer 
algebra system {\sf CHEVIE} \cite{chevie}, \cite{jmich}, for example. 
\end{exmp}

\begin{abs} {\bf Centralisers of semisimple elements.} \label{rem3} Let 
$C$ be an $F$-stable conjugacy class of semisimple elements of $\bG$. 
It is well-known that $C\cap \bT_0$ is non-empty and a single orbit under 
the action of $\bW$; furthermore, $C_\bG^\circ(t)$, for $t\in C\cap\bT_0$, 
is a connected reductive subgroup of the type considered above (see,
e.g., \cite[\S 3.5, \S 3.7]{Ca2}). Thus, we are led to consider the 
subset $\Xi^\circ\subseteq \Xi$ consisting of all pairs $(\Phi',w)\in\Xi$, 
for which there exists some $t\in \bT_0$ such that $\Phi'=\{\alpha \in
\Phi \mid \alpha(t)=1\}$ and $F(t)= \dot{w}^{-1}tw$. Then the 
$\bG^F$-conjugacy classes of subgroups of the form $C_\bG^\circ(s)$, 
where $s\in \bG^F$ is semisimple, are parametrised by the pairs in
$\Xi^\circ$ modulo the equivalence relation $\sim$ on~$\Xi^\circ$. (See 
again \cite{C1}, \cite{Dere}, \cite{Miz}.) Given $(\Phi',w)\in\Xi^\circ$, 
a corresponding semisimple element $s\in\bG^F$ is obtained as follows.
Let $t\in \bT_0$ be such that $\Phi'=\{\alpha \in\Phi\mid \alpha(t)=1\}$ 
and $F(t)=\dot{w}^{-1}tw$; then $C_\bG^\circ(t)=\langle \bT_0, \bU_\alpha
\; (\alpha \in \Phi') \rangle$. Let $g\in \bG$ be such that $g^{-1}F(g)=
\dot{w}$ and set $s:=gtg^{-1}$. Then $F(s)=s$ and $C_\bG^\circ(s)= 
gC_\bG^\circ(t) g^{-1}$.

Note that, if $d\in\bW$ has minimal length in the coset $\bW(\Phi')w$,
then $w=w'd$ for some $w'\in \bW(\Phi')$ and we still have $F(t)=
\dot{d}^{-1} t\dot{d}$ (since $\dot{w}'\in C_\bG(t)$). Hence, again, we 
may assume without loss of generality that $w=d$ and so the discussions
in Remarks~\ref{rem1}, \ref{rem2}, and Lemma~\ref{lem1} apply. The 
subsystems $\Phi'\subseteq \Phi$ which can arise at all as the root
system of $C_\bG^\circ(t)$, for some $t\in\bT_0$, are characterised in 
\cite[\S 2.3]{Dere}; given such a subset $\Phi'\subseteq \Phi$, the 
condition of whether there is some $w\in\bW$ such that $(\Phi',w)\in
\Xi^\circ$ may also depend on the isogeny type of $\bG$ and the 
$\F_q$-rational structure on $\bG$; see \cite[\S 5]{C1} and 
\cite[Chap.~2]{Dere} for further details.
\end{abs}

%%%%%%%%%%%%%%%%%%%%%%%%%%%%%%%%%%%%%%%%%%%%%%%%%%%%%%%%%%%%%%%%%%%%%%%%%%%
\section{On the evaluation of Deligne--Lusztig characters} \label{sec2}

Let $\bT\subseteq \bG$ be an $F$-stable maximal torus and $\theta
\in\Irr(\bT^F)$ be an irreducible character. Then we have a corresponding 
virtual character $\RTG(\theta)$ of $\bG^F$, as defined by
Deligne--Lusztig \cite{DeLu} (see also \cite[Chap.~7]{Ca2}). These 
virtual characters span a significant subspace of the space of all class 
functions on $\bG^F$ (see the introduction of \cite{LuB} and 
\cite[Cor.~2.7.13]{gema} for a more precise measure of what 
``significant'' means). It is known that, if $u\in\bG^F$ is unipotent, 
then $Q_\bT^\bG(u):=\RTG\bigl(\theta\bigr)(u)\in\Z$ does not depend 
on $\theta$; the function $u\mapsto Q_\bT^\bG(u)$ is called a {\it Green 
function}. We now have the following important character formula. 

Let $\tilde{g}\in\bG^F$ and write $\tilde{g}=su=us$, where $s\in \bG^F$ 
is semisimple and $u\in \bG^F$ is unipotent. Then, setting $\bH_s:=
C_{\bG}^\circ(s)$, we have 
\[ \RTG\bigl(\theta\bigr)(\tilde{g})=\frac{1}{|\bH_s^F|}\sum_{x\in\bG^F\,:
\,x^{-1}sx\in\bT} Q_{x\bT x^{-1}}^{\bH_s}(u) \,\theta(x^{-1}sx).\]
Note that, firstly, $\bH_s$ is connected and reductive; secondly, 
if $x^{-1} sx \in\bT^F$, then $x\bT x^{-1}$ is an $F$-stable maximal
torus contained in $\bH_s$; furthermore, $u$ is known to belong to $\bH_s$.
(See \cite[4.2]{DeLu} or \cite[7.2.8]{Ca2} for further details.) In
particular, the formula shows that all values of $\RTG(\theta)$ belong
to the field $\Q(\theta(t)\mid t\in\bT^F)$. We also see that, if $s$ is
not conjugate in $\bG^F$ to an element of $\bT^F$, then 
$\RTG\bigl(\theta\bigr)(g)=0$. 

We now explain how the above formula can be evaluated explicitly; for this
purpose, we need to \\
$\mbox{}\quad$ (1) know the values of the Green functions 
of $\bH_s$, \\
$\mbox{}\quad$ (2) deal with the sum over all $x\in\bG^F$ such that 
$x^{-1}sx\in\bT$.\\
As far as (1) is concerned, see the surveys in 
\cite[Chap.~13]{rDiMi} and \cite[\S 2.8]{gema}. In any case, for
$\bG$ simple of exceptional type, explicit tables are known by 
Beynon--Spaltenstein \cite{BeSp} and Shoji \cite{Sh82}. (The fact that
these tables remain valid whenever $p$ is a ``good'' prime for $\bG$ 
follows from \cite[Theorem~5.5]{S3}; see also \cite{ekay} for 
``bad''~$p$.) The tables can be obtained via the function {\tt ICCTable} 
of Michel's version of {\sf CHEVIE} \cite{jmich}. As far as (2) is 
concerned, the following result provides a first simplification.

\begin{lem}[See \protect{\cite[2.2.23]{gema}}] \label{lem2} In the above
setting, assume that $s$ is conjugate in $\bG^F$ to an element in $\bT^F$.
Let $\bT_1,\ldots,\bT_m$ be representatives of the $\bH_s^F$-conjugacy 
classes of $F$-stable maximal tori of $\bH_s$ that are conjugate in $\bG^F$ 
to~$\bT$. For each~$i$, let $\tilde{g}_i\in \bG^F$ be such that $\bT_i=
\tilde{g}_i\bT \tilde{g}_i^{-1}$. Then 
\[ \RTG\bigl(\theta\bigr)(\tilde{g})=\sum_{1\leq i\leq m} 
Q_{\bT_i}^{\bH_s}(u) \frac{1}{|\bW(\bH_s,\bT_i)^F|}\sum_{y\in
\bW(\bG,\bT_i)^F}\theta(\tilde{g}_i^{-1}\dot{y}^{-1}s\dot{y}\tilde{g}_i)\]
where $\bW(\bG,\bT_i)=N_{\bG}(\bT_i)/\bT_i$ and $\bW(\bH_s,\bT_i)=
N_{\bH_s}(\bT_i)/\bT_i$.
\end{lem}

Now note that the subgroup $\bT^F\subseteq \bG^F$ is not really 
computationally accessible, and the same is true for $\bT_1^F,\ldots, 
\bT_m^F$. All we can do explicitly are computations within $\bT_0$. 
Hence, in order to proceed, we use a different model for $\RTG(\theta)$, 
as already constructed in \cite{DeLu}. Assume that $\bT$ is of type 
$w\in\bW$ and let $g_1\in \bG$ be such that $g_1^{-1}F(g_1)=\dot{w}$. 
Then define
\[\bT_0[w]:=\{t\in\bT\mid F(t)=\dot{w}^{-1}t\dot{w}\}=g_1^{-1}\bT^Fg_1
\subseteq \bG.\]
Let $\theta'\in \Irr(\bT_0[w])$
be the irreducible character defined by $\theta'(t):=\theta(g_1tg_1^{-1})$
for all $t \in \bT_0[w]$. Then we have $\RTG(\theta)=R_w^{\theta'}$, where
the right hand side is constructed directly from $(w,\theta')$ (see 
\cite[2.3.18]{gema} for further details). Thus, the virtual characters 
$\RTG(\theta)$ may equally well be defined in terms of pairs $(w,\theta')$, 
where $w\in \bW$ and $\theta'\in \Irr(\bT_0[w])$~---~and the latter set
of pairs $(w, \theta')$ is, indeed, computationally accessible.

We can now apply the results in the previous section, especially
the discussions in Remark~\ref{rem2} and Lemma~\ref{lem1}.
Let $\tilde{g}=su=us\in \bG^F$ as above and $\bH_s:=C_\bG^\circ(s)$. Let 
$\bT'\subseteq \bH_s$ be a maximally split torus and $g\in \bG$ 
be such that $\bT'=g\bT_0g^{-1}$. Then $g^{-1}F(g)\in N_\bG(\bT_0)$ and 
we denote by $d\in \bW$ the image of $g^{-1}F(g)$ in $\bW$. Let $t:=
g^{-1}sg\in \bT_0$ and $\Phi':=\{\alpha \in\Phi\mid \alpha(t)=1\}$. Then 
$F(t)=\dot{w}^{-1}t\dot{w}$ and $(\Phi',d)\in \Xi^\circ$ parametrises 
the $\bG^F$-conjugacy class of~$\bH_s$; furthermore, we have that 
$d$ has minimal length in $\bW(\Phi')d$. As in Remark~\ref{rem1}, we 
define $\sigma' \in \mbox{Aut}(\bW)$ by $\sigma'(w)=d\sigma(w)d^{-1}$ 
for $w\in\bW$; then $\sigma'\bigl(\bW(\Phi')\bigr)=\bW(\Phi')$.
\begin{itemize}
\item[($*$)] Let $w_1',\ldots,w_m'\in \bW(\Phi')$ be representatives of 
the $\sigma'$-conjugacy classes of $\bW(\Phi')$ such that $w_i'd$ is 
$\sigma$-conjugate in $\bW$ to~$w$. For each $i$ let us fix an element
$x_i\in \bW$ such that $w=x_i^{-1}w_i'd\sigma(x_i)$. 
\end{itemize}
For $1\leq i\leq m$ let $h_i\in \bH_s$ be such that $h_i^{-1}F(h_i)=
g\dot{w}_i'g^{-1}$, and set $\bT_i:=h_i\bT'h_i^{-1}\subseteq \bH_d$. 
Then, by Lemma~\ref{lem1}, $\bT_1,\ldots,\bT_m$ are maximal tori as 
required in Lemma~\ref{lem2}. For $1\leq i \leq m$ define $\bT_0[w_i'd]
\subseteq \bT_0$ analogously to $\bT_0[w]$ above; then $\bT_0[w_i'd]=
\dot{x}_i\bT_0[w]\dot{x}_i^{-1}$ (see \cite[2.3.20]{gema}). So, given
$\theta'\in \Irr(\bT_0[w])$ as above, we can define a character $\theta_i'
\in \Irr(\bT_0[w_i'd])$ by
\[ \theta_i'(t):=\theta'(\dot{x}_i^{-1}t\dot{x}_i) \qquad \mbox{ for
$t\in \bT_0[w_i'd]$}.\]
Now let $C_{\bW,\sigma}(w)=\{x\in \bW\mid xw=w\sigma(x)\}$ be the 
$\sigma$-centraliser of $w$ in $\bW$; then $\dot{x}\bT_0[w]\dot{x}^{-1}=
\bT_0[w]$ for all $x\in C_{\bW,\sigma}(w)$. Defining $C_{\bW,\sigma}
(w_i'd)$ analogously, we have $\dot{x}\bT_0[w_i'd]\dot{x}^{-1}=\bT_0
[w_i'd]$ for all $x\in C_{\bW,\sigma}(w_i'd)$. Let also 
\[C_{\bW(\Phi'),\sigma'}(w_i')=\{x\in\bW(\Phi')\mid xw_i'=w_i'
\sigma'(x)\}= \bW(\Phi')\cap C_{\bW,\sigma}(w_i'd).\]
With this notation, we can now state the following result; see also 
L\"ubeck \cite[Satz~2.1]{lphd} for a slightly different formulation.

\begin{lem} \label{lem3} In the above setting, we have $\bW(\bH_s,
\bT_i)^F\cong C_{\bW(\Phi'),\sigma'}(w_i')$ and 
\[ \sum_{y\in \bW(\bG,\bT_i)^F} \theta(\tilde{g}_i^{-1}\dot{y}^{-1}s
\dot{y}\tilde{g}_i)=\sum_{c} \theta_i'(\dot{c}^{-1}t\dot{c}) \qquad 
\mbox{for $1\leq i \leq m$},\]
where $c$ runs over all elements of $C_{\bW,\sigma}(w_i'd)$. In 
particular, if $\theta=1_\bT$ is the trivial character of $\bT^F$, then 
the above sum equals $|C_{\bW,\sigma}(w_i'd)|$ for $1\leq i \leq m$.
\end{lem}

\begin{proof} Recall that $\bT=g_1\bT_0g_1^{-1}$, $\bT'=g\bT_0g^{-1}$ and
$\bT_i=h_i\bT'h_i^{-1}$ for all~$i$. Hence, setting $\tilde{g}_i:=h_i
g\dot{x}_ig_1^{-1}\in \bG$, we have $\tilde{g}_i\bT\tilde{g}_i^{-1}=\bT_i$
for all~$i$. Since $\bT$ and $\bT_i$ are $\bG^F$-conjugate, we can replace
$h_i$ by $h_it_i'$ for a suitable $t_i'\in \bT'$ such that $F(\tilde{g}_i)=
\tilde{g}_i$ (see the argument in the proof of \cite[Prop.~3.3.3]{Ca2}).
Thus, the elements $\tilde{g}_i$ are as required in Lemma~\ref{lem2}.
Next, by \cite[Prop.~3.3.6]{Ca2}, we have a group isomorphism
\[ C_{\bW,\sigma}(w_i'd)\stackrel{\sim}{\longrightarrow} \bW(\bG,\bT_i)^F,
\qquad c \mapsto h_ig\dot{c}g^{-1}h_i^{-1}.\]
(Recall that $h_ig\bT_0g^{-1}h_i^{-1}=\bT_i$ and $(h_ig)^{-1}F(h_ig)=
\dot{w}_i'\dot{d}$.) Hence, we have
\[\sum_{y\in \bW(\bG,\bT_i)^F} \theta(\tilde{g}_i^{-1}\dot{y}^{-1}s
\dot{y}\tilde{g}_i)=\sum_{c} \theta(\tilde{g}_i^{-1}h_ig\dot{c}^{-1}
g^{-1}h_i^{-1}sh_ig\dot{c}g^{-1}h_i^{-1}\tilde{g}_i)\]
where $c$ runs over all elements of $C_{\bW,\sigma}(w_i'd)$. Now
$g^{-1}h_i^{-1}sh_ig=t$; hence, the terms in the above sum on the right
hand side are given by
\[ \theta(\tilde{g}_i^{-1}h_ig\dot{c}^{-1}
t\dot{c}g^{-1}h_i^{-1}\tilde{g}_i)=\theta(g_1\dot{x}_i^{-1}\dot{c}^{-1}
t\dot{c}\dot{x}_ig_1^{-1})=\theta_i'(\dot{c}^{-1}t\dot{c})\]
for all $c\in C_{\bW,\sigma}(w_i'd)$, as required. Finally, consider
the assertion concerning $\bW(\bH_s,\bT_i)^F$. Let $\bW_s=N_{\bH_s}(\bT')
/\bT'$ and $\sigma_s\colon \bW_s\rightarrow \bW_s$ be induced by~$F$. Let
$\bH'= C_\bG^\circ(t)= g^{-1}\bH_sg$ and $F'\colon \bH'\rightarrow\bH'$ be 
as in Remark~\ref{rem1}; recall that $F'$ induces $\sigma'\in \mbox{Aut}
(\bW(\Phi'))$. As discussed in Remark~\ref{rem2}, conjugation by $g$ induces 
a bijection between the $\sigma'$-conjugacy classes in $\bW(\Phi')$
and the $\sigma_s$-conjugacy classes in~$\bW_s$. Thus, we have $\bW(\bH_s,
\bT_i)^F\cong \bW(\bH',\bT_i')^{F'}$ where $\bT_i'\subseteq \bH'$ is 
an $F'$-stable maximal torus of type $w_i'\in\bW(\Phi')$ (relative to~$F'$). 
Again by \cite[Prop.~3.3.6]{Ca2}, the group $\bW(\bH',\bT_i')^{F'}$ is
isomorphic to $C_{\bW(\Phi'),\sigma'}(w_i')$.
\end{proof}

The point about the above result is that the formula on the right hand
side of the identity can be explicitly and effectively computed, once a
character $\theta'\in\Irr(\bT_0[w])$ has been specified: all we need to
know is the action of $\bW$ on $\bT_0$, plus information concerning 
various $\sigma$-conjugacy classes in $\bW$.

\begin{exmp} \label{reguni} Assume that $\tilde{g}=su=us$ where 
$u\in\bH_s$ is regular unipotent. Then $Q_{\bT_i}^{\bH_s}(u)=1$ for 
$1\leq i \leq m$ (see \cite[Theorem~9.16]{DeLu}). Let $\theta=1_\bT$ be 
the trivial character of $\bT^F$. Then 
\[\RTG\bigl(1_\bT\bigr)(g)= |C_{\bW,\sigma}(w)|\sum_{1\leq i\leq m} 
|C_{\bW(\Phi'), \sigma'}(w_i')|^{-1}.\]
Indeed, this is now clear by Lemmas~\ref{lem2} and~\ref{lem3}. Note that, 
since the maximal tori $\bT$ and~$\bT_i$ are conjugate in $\bG^F$, we have 
$\bW(\bG,\bT)^F \cong \bW(\bG,\bT_i)^F$ and, hence,
$\bW(\bG,\bT_i)^F\cong C_{\bW,\sigma}(w)$ for $1\leq i\leq m$.
\end{exmp}

\begin{exmp} \label{greenp} Assume that $\bG$ is of split type; then
$\sigma=\mbox{id}_\bW$. Then we define
\[ R_\phi^\bG:=\frac{1}{|\bW|} \sum_{w\in \bW} \phi(w)R_{\bT_w}^\bG(1)
\qquad \mbox{for $\phi\in\Irr(\bW)$},\]
where $\bT_w\subseteq \bG$ is an $F$-stable maximal torus of 
type $w$ and $1$ stands for the trivial character of $\bT_w^F$. (This 
is a very special case of \cite[(3.7.1)]{LuB}.) We also have 
\[ R_{\bT_w}^\bG(1)=\sum_{\phi\in\Irr(\bW)} \phi(w)R_\phi^{\bG}
\qquad \mbox{for all $w\in\bW$};\]
so knowing the $R_{\bT_w}^\bG(1)$'s is equivalent to knowing the 
$R_\phi^\bG$'s. Also assume now that $\tilde{g}=su=us$ is such
that $s\in\bT_0^F$. Then one easily sees that
\[ R_\phi^\bG(\tilde{g})=\sum_{\psi\in\Irr(\bW_s)}m(\psi,\phi) 
R_{\psi}^{\bH_s}(u) \qquad\mbox{for any $\phi\in \Irr(\bW)$},\]
where $\bW_s=N_{\bH_s}(\bT_0)/\bT_0\subseteq \bW$ is the Weyl group
of $\bH_s$ and $m(\psi,\phi)$ denotes the multiplicity of $\psi
\in\Irr(\bW_s)$ in the restriction of~$\phi$. Thus, here the question 
of finding the fusion of $F$-stable maximal tori from $\bH_s$ to $\bG$ 
has been absorbed into the question of decomposing the restriction of 
any $\phi\in\Irr(\bW)$ to~$\bW_s$. 
\end{exmp}

\begin{exmp} \label{secf4g} Let $\bG$ be simple of type $F_4$, where 
$p\neq 2$. There exists an involution $s\in\bT_0^F$ such that $\bH':=
C_\bG(s)$ has a root system of type $B_4$ (see \S \ref{b4} below 
for further details). Furthermore, there is a unipotent element
$u\in \bH^F$ such that, if we let $\Sigma$ be the $\bG^F$-conjugacy 
class of $su$, then $\Sigma^F$ splits into five conjugacy classes 
in $\bG^F$, with centraliser orders $8q^8,8q^8,4q^4,4q^4,4q^4$ (see 
\S \ref{csha7} for more details). The condition on the centraliser 
orders uniquely determines the conjugacy class of~$u$ in~$\bH'$. Now 
let $u'\in \bH'^F$ be one of the unipotent elements such that $su'\in 
\Sigma$ and $|C_{\bG}(su')^F|=8q^8$. 

Let $\bW'\subseteq \bW$ be the Weyl group of $\bH'$. Then $\Irr(\bW')$
is parametrised by the bi-partitions of~$4$. By the output of the 
function {\tt ICCTable} in Michel's version of {\sf CHEVIE} \cite{jmich}, 
the only $\psi\in \Irr(\bW')$ such that $R_{\psi}^{\bH'}(u')\neq 0$ are 
$\psi_{(4,-)}$, $\psi_{(3,1)}$, $\psi_{(-,4)}$, $\psi_{(22,-)}$, 
$\psi_{(2,2)}$. Furthermore, we have 
\[ R_{\psi_{(4,-)}}^{\bH'}(u')=1, \; R_{\psi_{(3,1)}}^{\bH'}(u')=q,\; 
R_{\psi_{(-,4)}}^{\bH'}(u')=R_{\psi_{(22,-)}}^{\bH'}(u')=
R_{\psi_{(2,2)}}^{\bH'}(u')=q^2,\]
In order to evaluate the formula for $R_\phi^\bG$ in Example~\ref{greenp}, 
we need to know the multiplicities $m(\psi,\phi)$ for $\psi \in
\Irr(\bW')$ and $\phi\in \Irr(\bW)$; these are readily available via 
the function {\tt InductionTable} of {\sf CHEVIE} \cite{chevie}.
This yields the following values:
\begin{gather*}
R_{\phi_{9,6}''}^\bG(su')=R_{\phi_{4,8}}^\bG(su')=
R_{\phi_{1,12}''}^\bG(su')=q^2,\\ R_{\phi_{12,4}}^\bG(su')= 
R_{\phi_{9,6}'}^\bG(su')= R_{\phi_{1,12}'}^\bG(su')=0,
\end{gather*}
where we use the notation of Carter \cite[p.~413]{Ca2} for the irreducible 
characters of~$\bW$. ({\sf CHEVIE} uses the same notation; the conversion 
to the notation defined and used by Lusztig \cite{LuB} is displayed in 
Table~\ref{labelf4}.) The knowledge of the above values will turn out to 
be useful in the further discussion in \S \ref{csha7}.
\end{exmp}
%, where the ordering is as in the table in 
%\cite[p.413]{Ca2}.) 

\begin{table}[htbp] \caption{Conventions for the labelling of $\Irr(\bW)$
for type $F_4$} \label{labelf4}
\begin{center}
$\renewcommand{\arraystretch}{1.0} \renewcommand{\arraycolsep}{4pt} 
\begin{array}{l}
\begin{array}{|ccccccccccccc|}  \hline
\phi_{1,0}& \phi_{1,12}''& \phi_{1,12}'& \phi_{1,24}& 
\phi_{2,4}''& \phi_{2,16}'& \phi_{2,4}'& \phi_{2,16}''& \phi_{4,8}& 
\phi_{9,2}& \phi_{9,6}''& \phi_{9,6}' & \phi_{9,10}\\
1_1& 1_3& 1_2& 1_4& 2_3& 2_4& 2_1& 2_2& 4_1& 9_1& 9_3& 9_2 & 9_4\\ 
\hline\end{array}
\\[10pt] \begin{array}{|ccccccccccccc|}  \hline
\phi_{6,6}'& \phi_{6,6}''& \phi_{12,4}& \phi_{4,1}& \phi_{4,7}''& 
\phi_{4,7}'& \phi_{4,13}& \phi_{8,3}''& \phi_{8,9}'& \phi_{8,3}'& 
\phi_{8,9}''& \phi_{16,5}& \\ 6_1& 6_2& 12_1& 4_2& 4_4& 4_3& 4_5& 
8_3& 8_4& 8_1& 8_2& 16_1& \\ \hline \end{array}\\
\text{\footnotesize The labels $\phi_{1,0}$ etc.\ are those in 
\cite[p.~413]{Ca2}; the labels $1_1$ etc.\ those of Lusztig 
\cite[4.10]{LuB}.}
\end{array}$
\end{center}
\end{table}

%%%%%%%%%%%%%%%%%%%%%%%%%%%%%%%%%%%%%%%%%%%%%%%%%%%%%%%%%%%%%%%%%%%%%%%%%%%
\section{Characteristic functions and conjugacy classes} \label{sec3}

Let $\hat{\bG}$ denote the set of character sheaves on $\bG$ (up to
isomorphism), as defined by Lusztig \cite{L2a}. If $A\in \hat{\bG}$ is
$F$-invariant, that is, we have $F^*A\cong A$, then the choice of an 
isomorphism $\phi\colon F^*A\stackrel{\sim}{\rightarrow} A$ gives rise 
to a characteristic function $\chi_{A}\colon \bG^F\rightarrow 
\overline{\Q}_l$ (where $l\neq p$ is a prime); see \cite[\S 5]{Lintr}. The 
isomorphism $\phi$ can be chosen such that the values of $\chi_{A}$ 
are cyclotomic integers and the standard inner product of $\chi_{A}$
with itself is~$1$. Hence, we may assume that $\chi_{A}$ is a 
function with values in~$\K$. The various functions arising in this
way form an orthonormal basis of the space of class functions on~$\bG^F$. 
(See \cite[\S 25]{L2e}; these results hold unconditionally because of
the ``cleanness'' established in \cite{L10}.) Similarly to the
situation for $\Irr(\bG^F)$, we have a partition $\hat{\bG}=\coprod_s 
\hat{\bG}_s$ where $s$ runs over the semisimple elements (up to
conjugation) in a group $\bG^*$ dual to $\bG$ (see \cite[1.2]{L7}.) 
The character sheaves in~$\hat{\bG}_s$, where $s=1$, are called
unipotent character sheaves.

In the case where $A$ is a cuspidal character sheaf (and $\bG$ is
simple), the characteristic functions $\chi_{A}$ can be evaluated 
in a simple way. We begin with some general remarks concerning conjugacy 
classes.

\begin{abs} {\bf Parametrisation of $\bG^F$-conjugacy classes.} 
\label{paracl} Let $\Sigma$ be an $F$-stable conjugacy class of $\bG$. 
Let us fix a representative $g_1\in\Sigma^F$ and set $A_\bG(g_1):=
C_\bG(g_1)/C_\bG^\circ(g_1)$. Then $F$ induces an automorphism of 
$A_\bG(g_1)$ that we denote by the same symbol. Given $a\in A_\bG(g_1)$,
let $\dot{a}\in C_\bG(g_1)$ be a representative of~$a$ and write
$\dot{a}=x^{-1}F(x)$ for some $x\in \bG$. Then $g_a:=xg_1x^{-1}\in\bG^F$; 
let $C_a$ be the $\bG^F$-conjugacy class of~$g_a$. A standard argument 
(using Lang's Theorem, see \cite[I, 2.7]{SS}) shows that $C_a$ only 
depends on~$a$; furthermore $\Sigma^F=\bigcup_{a\in A_\bG(g_1)} C_a$, 
where $C_a=C_{a'}$ if and only if $a,a'$ are $F$-conjugate in $A_\bG(g_1)$.
Now there are two natural operations on the $\bG^F$-conjugacy
classes contained in $\Sigma^F$.

(a) The first one is denoted by $\mbox{Sh}_\bG$ and called the Shintani
map. Let $C$ be a $\bG^F$-conjugacy class contained in $\Sigma^F$; 
thus, $C=C_a$ for some $a\in A_\bG(g_1)$. Let $g\in C$ and write 
$g=x^{-1}F(x)$ for some $x\in \bG$. Then $g':=xgx^{-1}\in \bG^F$ and 
the $\bG^F$-conjugacy class of $g'$ does not depend on the choice 
of~$g$ or~$x$; we denote that class by $\mbox{Sh}_\bG(C)$. By 
Digne--Michel \cite[Chap.~IV, Prop.~1.1]{DiMi0}, we have 
\[ \mbox{Sh}_\bG(C_a)=C_{\bar{g}_1a} \qquad (a\in A_\bG(g_1))\]
where $\bar{g}_1$ denotes the image of $g_1\in C_\bG(g_1)$ in
$A_\bG(g_1)$. 

%For example, if $A_\bG(g_1)=\langle \bar{g}_1\rangle$,
%then $A_\bG(g_1)$ is abelian and $F$ acts trivially on it; then the 
%above formula shows that $\mbox{Sh}_\bG$ permutes the classes inside 
%$\Sigma^F$ transitively. 

(b) The second one, defined in \cite[3.1]{L7}, only plays a role when 
$Z(\bG)\neq \{1\}$. Let $z\in Z(\bG)$ and write $z=t^{-1}F(t)$ for 
some $t\in\bG$. (We could even take $t\in \bT_0$.) As above, let $C=C_a$ 
be a $\bG^F$-conjugacy class contained in $\Sigma^F$. One easily sees
that $\gamma_z(C):=tCt^{-1}$ is a conjugacy class in $\bG^F$ and does 
not depend on the choice of~$t$; furthermore, we have
\[ \gamma_z(C_a)=C_{\bar{z}a}  \qquad (a\in A_\bG(g_1))\]
where $\bar{z}\in A_\bG(g_1)$ denotes the image of $z$ under the
natural map $Z(\bG)\rightarrow A_\bG(g_1)$.
\end{abs}

\begin{abs} {\bf Characteristic functions of cuspidal character sheaves.} 
\label{charf} Assume that $\bG$ is a simple algebraic group. Let $A$ 
be a cuspidal character sheaf on $\bG$ such that $F^*A\cong A$. (See 
\cite[Def.~3.10]{L2a}; such an $A$ may be unipotent or not.) Then there
exists an $F$-stable conjugacy class $\Sigma$ of $\bG$ and an irreducible, 
$\bG$-equivariant $\overline{\Q}_l$-local system $\cE$ on $\Sigma$ such 
that $F^*\cE\cong \cE$ and $A=\mbox{IC}(\overline{\Sigma},\cE)[\dim 
\Sigma]$; see \cite[3.12]{L2a}. Let us fix $g_1\in \Sigma^F$ and set 
$A_\bG(g_1):=C_\bG(g_1)/C_\bG^\circ(g_1)$, as above. We further assume that:
\begin{itemize}
\item[($*$)] the local system $\cE$ is one-dimensional and, 
hence, corresponds to an $F$-invariant linear character $\psi\colon 
A_\bG(g_1) \rightarrow \K^\times$ (via \cite[19.7]{Ldisc4}).
\end{itemize}
(This assumption will be satisfied in all examples that we consider.)
Now ($*$) implies that the function $\psi\colon A_\bG(g_1) \rightarrow 
\K^\times$ is constant on the $F$-conjugacy classes of $A_\bG(g_1)$. 
Hence, we obtain a class function $\chi_{g_1,\psi}\colon \bG^F\rightarrow
\K$ by setting
\[\chi_{g_1,\psi}(g):=\left\{\begin{array}{cl} q^{(\dim \bG-\dim\Sigma)/2}
\psi(a) & \quad\mbox{if $g\in C_a$ for some $a\in A_\bG(g_1)$},\\ 0 & \quad
\mbox{if $g\not\in \Sigma^F$}.\end{array}\right.\]
(Note that there are cases where $\dim \bG-\dim\Sigma$ is not even; in
such a case, we also need to fix a square root of $q$ in $\K$.) 
Since $\cE$ is one-dimensional, we can choose an isomorphism 
$F^*\cE \stackrel{\sim}{\rightarrow}\cE$ such that the induced map on the 
stalk $\cE_{g_1}$ is given by scalar multiplication by $q^{(\dim \bG-
\dim \Sigma)/2}$. Then this isomorphism canonically induces an isomorphism
$\phi\colon F^*A \stackrel{\sim}{\rightarrow} A$ and $\chi_{g_1,\psi}$ is 
the corresponding characteristic function $\chi_A$, of norm~$1$ with 
respect to the standard inner product. (This follows from the fact that
$A$ is ``clean'' \cite{L10}, using the construction in 
\cite[19.7]{Ldisc4}.) We shall also set
\[ \lambda_A:=\psi(\bar{g}_1)\qquad\mbox{where $\bar{g}_1$ denotes the 
image of $g_1$ in $A_\bG(g_1)$}.\]
Then $\lambda_A$ is a root of unity that only depends on $A$ (see 
Shoji \cite[Theorem~3.3]{S2}, \cite[Prop.~3.8]{S2}); it is a useful 
invariant of $A$. In this context, we have the following basic problem,
formulated by Lusztig \cite[0.4(a)]{L7}:
\begin{itemize}
\item[($\clubsuit$)] Express the functions $\chi_{g_1,\psi}$ as explicit
linear combinations of $\Irr(\bG^F)$.
\end{itemize}
This problem is solved in many cases, but not in complete generality.
Some examples in small rank cases (types $A_1$, $C_2$, ${^3\!D}_4$, 
$\ldots$) are mentioned in \cite[Example~7.8]{mylaus}. We will produce 
further examples below.
\end{abs}

For $\bG$ simple of exceptional type, many cuspidal character sheaves turn 
out to be unipotent. (Exceptions only occur in types $E_6$ and $E_7$.)
So it is of particular importance to address theses cases.

\begin{abs} {\bf Cuspidal unipotent character sheaves.}  \label{ualm} 
Assume that $\bG$ is simple, of split type (so $\sigma=\mbox{id}_\bW$). 
Let $\Unip(\bG^F)$ denote the set of unipotent characters of $\bG^F$.
By \cite[Main Theorem~4.23]{LuB}, $\Unip(\bG^F)$ is parametrised
by a certain set $X(\bW)$ which only depends on $\bW$. For each $x\in
X(\bW)$, we have a corresponding almost character $R_x$, defined as an 
explicit linear combination of $\Unip(\bG^F)$; see \cite[4.24.1]{LuB}. 
There is an embedding $\Irr(\bW)\hookrightarrow X(\bW)$, $\phi\mapsto 
x_\phi$, such that
\[ R_{x_\phi}=R_\phi^\bG=\frac{1}{|\bW|} \sum_{w\in \bW} \phi(w)
R_{\bT_w}^\bG(1) \qquad \mbox{for $\phi\in\Irr(\bW)$}.\]
Thus, the values of $R_{x_\phi}$ can be computed using the 
character table of $\bW$ and the results discussed in 
Section~\ref{sec2}. Finally, the unipotent character sheaves on
$\bG$ are also parametrised by $X(\bW)$; see \cite[Theorem~23.1]{L2e}
(plus the ``cleanness'' in~\cite{L10}). If $x\in X(\bW)$ is such that 
$A_x$ is cuspidal and $F$-invariant, then we have a corresponding 
characteristic function $\chi_{g_1,\psi}$ as in \S \ref{charf}. 
In this situation, the solution of ($\clubsuit$) in \S \ref{charf}
is known, that is, for all $x\in X(\bW)$ such that $A_x$ is cuspidal, 
we have 
\begin{itemize}
\item[($\clubsuit^\prime$)] $R_x=\zeta\chi_{g_1,\psi}$ for 
some scalar $\zeta\in\K$ of absolute value $1$.
\end{itemize}
If $p$ is sufficiently large, then this is part of Lusztig 
\cite[Theorem~0.8]{L7}. For arbitrary $p$, this is part of the main 
results of Shoji \cite{S2}, \cite{S3}, (The latter results hold 
without condition on~$p$, thanks to the ``cleanness'' in \cite{L10}.) 
The scalars $\zeta$ are determined by Shoji \cite{Sclass}, \cite{S7} for 
$\bG$ of classical type. For exceptional types, there are a number of 
cases where the scalars $\zeta$ remain to be determined, and it is one 
purpose of this paper to reduce the number of open cases.

The following technical result will be needed in Section~\ref{sece7}.
\end{abs}

\begin{lem} \label{samev} In the setting of \S \ref{paracl},
let $a\in A_\bG(g_1)$ and $z\in Z(\bG)$. Then every $\rho\in\Unip(\bG^F)$ 
takes the same value on $C_a$ and on $C_{\bar{z}a}$.
\end{lem}

\begin{proof} Let $g\in C_a$. By \S \ref{paracl}(b) we have 
$C_{\bar{z}a}=\gamma_z(C_a)$. Hence, writing $z=t^{-1}F(t)$ for 
some $t\in \bG$, we have $g':=tgt^{-1}\in C_{\bar{z}a}$. Let 
$\rho\in\Unip(\bG^F)$. In order to show that $\rho(g)=\rho(g')$,
we use a regular embedding $\bG\subseteq \tilde{\bG}$ (see, e.g., 
\cite[\S 1.7]{gema}). Thus, $\tilde{\bG}$ is  a connected reductive
group with a connected center and $\bG$, $\tilde{\bG}$ have the same
derived subgroup; furthermore, $\tilde{\bG}$ is also defined over
$\F_q$ and we denote the corresponding Frobenius map again by~$F$.
Now $Z(\bG)\subseteq Z(\tilde{\bG})$ and so, since $Z(\tilde{\bG})$
is connected, we can write $z=\tilde{t}^{-1}F(\tilde{t})$ where
$\tilde{t}\in Z(\tilde{\bG})$. Then $h:=t\tilde{t}^{-1}\in 
\tilde{\bG}^F$ and so $hgh^{-1}=t\tilde{t}^{-1}g\tilde{t}t^{-1}=
tgt^{-1}=g'$, that is, $g$ and $g'$ are conjugate in $\tilde{\bG}^F$.
Since $\rho$ is unipotent, there exists some $F$-stable maximal
torus $\bT\subseteq \bG$ such that $\rho$ occurs in $R_{\bT}^{\bG}(1)$
(where $1$ stands for the trivial character of~$\bT^F$). There is
an $F$-stable maximal torus $\tilde{\bT}\subseteq \tilde{\bG}$
such that $\bT\subseteq \tilde{\bT}$. Since $R_\bT^\bG(1)$ is
the restriction of $R_{\tilde{\bT}}^{\tilde{\bG}}(1)$ to $\bG^F$
(see \cite[Remark~2.3.16]{gema}), there exists some $\tilde{\rho}
\in \Unip(\tilde{\bG})$ such that $\rho$ occurs in the restriction of 
$\tilde{\rho}$ to~$\bG^F$. But it is known that $\tilde{\rho}|_{\bG^F}$
is irreducible (see \cite[Lemma~2.3.14]{gema}) and so $\rho$ is equal 
to the restriction of~$\tilde{\rho}$. Thus, we certainly have
$\rho(g)=\tilde{\rho}(g)=\tilde{\rho}(g')=\rho(g')$.
\end{proof}

\begin{exmp} \label{g2csh} Let $\bG$ be simple of type $G_2$. In this
case, the complete character table of $\bG^F$ is known; see
Chang--Ree \cite{chre} ($p\neq 2,3$), Enomoto \cite{En1} ($p=3$)
and Enomoto--Yamada \cite{En2} ($p=2$). Now, there are four cuspidal
character sheaves, and they are all unipotent; see \cite[\S 20]{L2d}, 
\cite[\S 6, \S 7]{S2}. From the known character tables, 
the above identities ($\clubsuit^\prime$) and the required scalars 
$\zeta$ can be easily extracted. For example, if $p\neq 2,3$, then the 
four functions $Y_1,Y_2,Y_3,Y_4$ printed on \cite[p.~411]{chre} are 
characteristic functions of the four cuspidal character sheaves on $\bG$.
\end{exmp}

Let us go back to the general case. Implicit in ($\clubsuit$) and
($\clubsuit^\prime$) is the problem of choosing a convenient
representative $g_1\in\Sigma^F$. In a number of cases, $\Sigma$ 
consists of regular elements in $\bG$. In such a case, there are
additional techniques to single out a canonical choice for $g_1 \in 
\Sigma^F$; see Corollary~\ref{correg1} below. 

\begin{abs} {\bf Regular elements.} \label{remreg}  An element $g\in \bG$ 
is called regular, if $\dim C_\bG(g)$ is as small as possible; it is known 
that this is equivalent to the condition that $\dim C_\bG(g)=\dim \bT_0$. 
Furthermore, let $g=su=us$ be the Jordan decomposition of $g$ (where $s$ 
is semisimple and $u$ is unipotent). Then $g$ is regular if and only if
$u$ is regular in $C_{\bG}^\circ(s)$. By Steinberg \cite[Theorem~1.2]{St},
every semisimple element of $\bG$ is the semisimple part of some 
regular element; finally, two regular elements of $\bG$ are conjugate if 
and only if their semisimple parts are conjugate. In particular, all 
regular unipotent elements are conjugate. 

Assume now that $\bG$ is simple and simply connected. Then a cross-section
for the conjugacy classes of regular elements has been found by Steinberg 
\cite{St}. Let us write $\bB=\bU\bT_0$ where $\bU$ is the unipotent 
radical of $\bB$. Let $\bB^-\subseteq \bG$ be the opposite Borel 
subgroup; then $\bB^-=\bU^-\bT_0$, where $\bU^-$ is the unipotent 
radical of $\bB^-$, and we have $\bU\cap \bU^-=\{1\}$. Let $w_c:=
w_{\alpha_1}\cdots w_{\alpha_r} \in\bW$ be a Coxeter element, where 
$r=\dim \bT_0$ and $\alpha_1, \ldots,\alpha_r$ is a fixed enumeration of 
the simple roots in~$\Phi^+$. Then the required cross-section is given by 
\[ \cN_{\dot{w}_c}:=\bU\dot{w}_c\cap \dot{w}_c\bU^-\;\subseteq \;
\bU \dot{w}_c \bU\;\subseteq \;\bB \dot{w}_c \bB.\]
Indeed, by \cite[Theorem~1.4, Lemma~7.3]{St}, all elements of 
$\cN_{\dot{w}_c}$ are regular and every regular element of $\bG$ is 
conjugate to exactly one element in $\cN_{\dot{w}_c}$. And everything
takes place inside the single double coset $\bU \dot{w}_c \bU$; note 
that this depends on the choice of the representative
$\dot{w}_c$ of~$w_c\in\bW$. The following result is a very special case
of the results on ``$C$-small'' classes in \cite[\S 5]{Lfrom}, so 
we include the proof here. (It is already mentioned in the proof of
\cite[Lemma~8.10]{L3}.) 
\end{abs}

\begin{prop}[Steinberg, He--Lusztig] \label{lemreg} Assume that $\bG$ is 
simple and simply connected. Let $\Sigma$ be a $\bG$-conjugacy class of 
regular elements.
\begin{itemize}
\item[(a)] The set $\Sigma \cap \bU\dot{w}_c\bU\neq \varnothing$ 
is a single $\bU$-orbit (under conjugation).
\item[(b)] The set $\Sigma \cap \bB\dot{w}_c\bB\neq \varnothing$ 
is a single $\bB$-orbit (under conjugation).
\item[(c)] In {\rm (a)}, the stabilisers  are trivial;
in {\rm (b)} they are equal to $Z(\bG)$.
\end{itemize}
\end{prop}

\begin{proof} By He--Lusztig \cite[Theorem~3.6(ii)]{luhe}, the map
\[ \bU\times \cN_{\dot{w}_c} \rightarrow \bU\dot{w}_c\bU,\qquad
(u,z) \mapsto uzu^{-1},\]
is bijective. (A closely related result is stated in \cite[Prop.~8.9]{St},
but since the proof is omitted there, we cite \cite{luhe}; see also 
\cite[\S 10]{Bo2}.) Let us denote by $g$ the unique element in 
$\Sigma\cap\cN_{\dot{w}_c}$; in particular, $g\in\Sigma\cap
\bU\dot{w}_c\bU$. 

(a) Given any $g'\in\Sigma\cap \bU\dot{w}_c\bU$, 
we can write $g'=uzu^{-1}$ where $u\in\bU$ and $z\in \cN_{\dot{w}_c}$. 
Thus, the two elements $z$ and $g$ in $\cN_{\dot{w}_c}$ are conjugate 
in~$\bG$. But then we must have $z=g$ and so $g'$ is conjugate to 
$g$ under~$\bU$. 

(b) Take any $g'\in\Sigma\cap \bB\dot{w}_c\bB$. Since $\bB\dot{w}_c
\bB=\bU\bT_0\dot{w}_c\bU$, we can write $g'=u_1t\dot{w}_cu_2$ 
where $u_1,u_2\in\bU$ and $t\in\bT_0$. By \cite[Lemma~7.6]{St}, we can 
further write $t\dot{w}_c=\tilde{t}\dot{w}_c\tilde{t}^{-1}$ for some 
$\tilde{t}\in\bT_0$. Then 
\[\tilde{t}^{-1}g'\tilde{t}=\tilde{t}^{-1}u_1t\dot{w}_cu_2\tilde{t}=
(\tilde{t}^{-1}u_1\tilde{t})\dot{w}_c (\tilde{t}^{-1}u_2\tilde{t})\in 
\bU\dot{w}_c\bU.\]
So we have $\tilde{t}^{-1}g'\tilde{t}=uzu^{-1}$ where $u\in\bU$ 
and $z\in \cN_{\dot{w}_c}$. Thus, $z\in\cN_{\dot{w}_c}$ and $g\in 
\cN_{\dot{w}_c}$ are conjugate in~$\bG$ and so $z=g$. Hence, 
$g'$ is conjugate to $g$ under~$\bB$. 

(c) The bijectivity of the above map $\bU\times \cN_{\dot{w}_c} 
\rightarrow \bU\dot{w}_c \bU$ immediately implies that $C_{\bU}
(g)=\{1\}$; thus, the stabilisers are trivial in~(a). For (b), we 
must show that $\mbox{Stab}_{\bB}(g)=Z(\bG)$. So let $b\in 
\bB$ be such that $bgb^{-1}=g$. Writing $g=v\dot{w}_c$ (where $v\in
\bU$) and $b=ut$ (where $u\in\bU$ and $t\in\bT_0$), we obtain
\[ v\dot{w}_c=g=bgb^{-1}=utv\dot{w}_ct^{-1}u^{-1}=(utvt^{-1})\dot{w}_c
\bigl(\dot{w}_c^{-1}t\dot{w}_ct^{-1}\bigr) u^{-1}.\]
Setting $\tilde{t}:=\dot{w}_c^{-1}t\dot{w}_ct^{-1}\in\bT_0$, we see 
that the left hand side lies in the double coset $\bU\dot{w}_c\bU$, 
and the right hand side lies in the double coset $\bU\dot{w}_c
\tilde{t}\bU$. But then the sharp form of the Bruhat decomposition 
implies that $\tilde{t}=1$ and so $t=\dot{w}_c^{-1}t\dot{w}_c$.
By \cite[Remarks~7.7(b)]{St}, this forces $t\in Z(\bG)$. But then 
$u\in C_{\bU}(g)$ and so $u=1$. Hence, $\mbox{Stab}_{\bB}(g) 
\subseteq Z(\bG)$; the reverse inclusion is clear.
\end{proof}

\begin{cor} \label{correg1} In the setting of Proposition~\ref{lemreg}, 
assume that $\Sigma$ is $F$-stable and $F(\dot{w}_c)=\dot{w}_c$. 
Then there exists a unique $\bG^F$-conjugacy class $C\subseteq 
\Sigma^F$ such that $C\cap \bU^F\dot{w}_c\bU^F\neq \varnothing$.
Furthermore, we have $C\cap \cN_{\dot{w}_c}\neq \varnothing$.
\end{cor}

\begin{proof} By Proposition~\ref{lemreg}, the group $\bU$ acts 
transitively on $X:=\Sigma\cap \bU\dot{w}_c\bU$ by conjugation,
and we have $\mbox{Stab}_{\bU}(x)=\{1\}$ for all $x\in X$; in 
particular, $\mbox{Stab}_{\bU}(x)$ is connected. A standard
application of Lang's Theorem (see, e.g., \cite[Prop.~1.4.9]{gema}) 
shows that $X^F\neq \varnothing$ and that $X^F$ is a single 
$\bU^F$-orbit. Thus, there exists a unique $\bG^F$-conjugacy class 
$C\subseteq \Sigma^F$ such that $C\cap (\bU\dot{w}_c\bU)^F\neq
\varnothing$. Note that $(\bU \dot{w}_c\bU)^F=\bU^F\dot{w}_c
\bU^F$, by the sharp form of the Bruhat decomposition.

Finally note that, since $F(\dot{w}_c)=\dot{w}_c$, we have 
$F(\cN_{\dot{w}_c})=\cN_{\dot{w}_c}$. Let $g$ be the unique element 
in $\Sigma\cap\cN_{\dot{w}_c}$. Then we also have $F(g)\in\Sigma 
\cap \cN_{\dot{w}_c}$ and so $F(g)=g$. Hence, $g\in \Sigma^F$ and 
$g\in \bU^F\dot{w}_c\bU^F$. So $C$ must be the $\bG^F$-conjugacy 
class of~$g$.
\end{proof}

\begin{exmp} \label{correg2} Let $\bG$, $\Sigma$, $C$ be as
in Corollary~\ref{correg1}. Then, clearly, we also have $C\cap 
\bB^F\dot{w}_c\bB^F\neq\varnothing$. Since the stabilisers
for the action of $\bB$ on $\Sigma\cap \bB\dot{w}_c\bB$ are
equal to $Z(\bG)$, the set $(\Sigma\cap \bB\dot{w}_c\bB)^F$ will
split into finitely many $\bB^F$-orbits, indexed by the
$F$-conjugacy classes of $Z(\bG)$. More precisely, let $g$ be the 
unique element in $C\cap \cN_{\dot{w}_c}$. Let $z_1,\ldots,z_r\in 
Z(\bG)$ be representatives of the $F$-conjugacy classes of $Z(\bG)$, 
where $z_1=1$. For $1\leq i\leq r$, we set $g_i:=t_igt_i^{-1}$, where
$t_i\in\bT_0$ is such that $z_i=t_i^{-1}F(t_i)$; here, we also 
assume $t_1=1$. Then $g_1,\ldots,g_r$ are representatives of the 
$\bB^F$-orbits on $(\Sigma\cap \bB\dot{w}_c\bB)^F$ (see 
\cite[I, 2.7]{SS}). Hence,
\begin{equation*}
(\Sigma\cap \bB\dot{w}_c\bB)^F=\bigl(C_1\cap \bB^F\dot{w}_c\bB^F
\bigr) \cup \ldots \cup \bigl(C_r\cap \bB^F\dot{w}_c\bB^F\bigr),\tag{a}
\end{equation*}
where $C_i$ is the $\bG^F$-conjugacy class of $g_i$, for all~$i$. Since
$z_i=t_i^{-1}F(t_i)\in Z(\bG)$, the map $C_\bG(g)^F \rightarrow 
C_\bG(g_i)^F$, $x\mapsto t_ixt_i^{-1}$, is bijective. In particular, we 
have 
\begin{equation*}
|C|=|C_i| \qquad \mbox{for $1\leq i \leq r$}.\tag{b}
\end{equation*}
The union in (a) may not be disjoint; but $C_i=C_j$ can only happen if 
the images of $z_i$ and $z_j$ in $A_\bG(g)$ are $F$-conjugate. Now, it 
is known that $A_\bG(g)$ is abelian, see \cite[III, 1.16 and 1.17]{SS}; 
so, for example, if the natural map $Z(\bG)\rightarrow A_\bG(g)$ is 
injective and $F$ acts trivially on $A_\bG(g)$, then the $C_i$ are all 
distinct. (This will cover most examples that we consider.) Finally,
we note the implication:
\begin{equation*}
Z(\bG)^F=\{1\}\qquad\Rightarrow\qquad(\Sigma\cap \bB\dot{w}_c\bB)^F=
C\cap \bB^F\dot{w}_c\bB^F.\tag{c}
\end{equation*}
Indeed, if $Z(\bG)^F=\{1\}$, then all elements of $Z(\bG)$ are
$F$-conjugate and so $r=1$.
\end{exmp}

\begin{lem} \label{lemreg3} Let $\bG$, $\Sigma$, $C$ be as
in Corollary~\ref{correg1}. Assume, furthermore, that $\bG$ is of 
split type, that $\Sigma=\Sigma^{-1}$, and that the union in
Example~\ref{correg2}(a) is disjoint.
\begin{itemize}
\item[(a)] There is a permutation $i\mapsto i'$ (of order~$2$) of the
set $\{1,\ldots,r\}$ such that $C_i^{-1}=C_{i'}$ for $1\leq i\leq r$.
\item[(b)] If $r$ is odd (e.g., if $Z(\bG)^F=\{1\}$), then we have 
$C_i= C_i^{-1}$ for $1\leq i\leq r$. 
\end{itemize}
\end{lem}

\begin{proof} First note that $C_i^{-1}\subseteq \Sigma^{-1}=\Sigma$
for $1\leq i \leq r$. We claim that 
\[C_i^{-1}\cap \bB^F\dot{w}_c\bB^F \neq \varnothing \qquad\mbox{for 
all~$i$}.\]
To see this, it is enough to show that $C_i\cap \bB^F \dot{w}_c^{-1}
\bB^F \neq \varnothing$. Now note that $w_c^{-1}$ also is a Coxeter
element (of minimal length); it is well-known that $w_c$ and $w_c^{-1}$ 
are conjugate in $\bW$. By \cite[0.2]{Lfrom}, \cite[Cor.~3.7(a)]{carp}, 
we have
\[ |C_i\cap \bB^F\dot{w}\bB^F|=|C_i\cap \bB^F\dot{w}'\bB^F|\]
for any two elements $w,w'\in\bW$ that are conjugate in $\bW$ and of 
minimal length in their conjugacy class. In particular, we can 
conclude that 
\[|C_i^{-1}\cap \bB^F\dot{w}_c\bB^F|=|C_i\cap \bB^F\dot{w}_c^{-1} 
\bB^F| =|C_i\cap \bB^F\dot{w}_c \bB^F|\neq 0\]
for all~$i$, as required. Thus, the above claim is proved. 

(a) Let $i\in\{1,\ldots,r\}$. By the above claim, $C_i^{-1}$ is a 
$\bG^F$-conjugacy class that is contained in $(\Sigma\cap \bB\dot{w}_c
\bB)^F$. So, by Example~\ref{correg2}(a), we must have $C_i^{-1}=
C_{i'}$ for some $i'\in\{1,\ldots,r\}$. Since the $C_i$ are all distinct, 
$i'$ is uniquely determined by~$i$; furthermore, the map $i\mapsto i'$
is a permutation (of order~$2$) of the set $\{1,\ldots,r\}$. 

(b) Assume that $r$ is odd. Then there must be some $i_0\in\{1,\ldots,
r\}$ such that $i_0=i_0'$, that is, $C_{i_0}=C_{i_0}^{-1}$. Now recall 
that $C_{i_0}$ is the $\bG^F$-conjugacy class of $g_{i_0}$, where 
$g_{i_0}=t_{i_0}gt_{i_0}^{-1}$ and $t_{i_0} \in \bT_0$ is such that
$z_{i_0}=t_{i_0}^{-1}F(t_{i_0})$. There exists some $x\in\bG^F$ such 
that $g_{i_0}^{-1}=xg_{i_0}x^{-1}$. It follows that $g^{-1}=ygy^{-1}$, 
where $y:=t_{i_0}^{-1}xt_{i_0}$. Now $F(y)=F(t_{i_0})^{-1}xF(t_{i_0})=
z_{i_0}^{-1}t_{i_0}^{-1}xt_{i_0}z_{i_0}=y$, since $z_{i_0}\in Z(\bG)$. 
Thus, we have shown that $C=C^{-1}$. Then the same argument, applied
to any $i\in\{1,\ldots,r\}$, also yields that $C_i=C_i^{-1}$.
\end{proof}

The following example (pointed out to the author by G. Malle) shows 
that the situation is really different when $r$ is even.

\begin{exmp} \label{regsl2} Let $\bG=\mbox{SL}_2(k)$ with $p\neq 2$.
Then $Z(\bG)$ has order $2$ and so $r=2$. Let $\Sigma$ be the class 
of regular unipotent elements; we certainly have $\Sigma=\Sigma^{-1}$.
Let $\bB$ be the Borel subgroup consisting of the upper triangular 
matrices in $\bG$; also let $\bW=\langle s_1\rangle$. Then one checks that 
\[\renewcommand{\arraystretch}{0.8} g:=\left(\begin{array}{cc} \; 1 \; & 
\; 0\;  \\ 1 & 1\end{array}\right) \in\Sigma^F \cap \bU\dot{s}_1\bU^F
\qquad\mbox{where} \qquad \dot{s}_1:=\left(\begin{array}{rr} 0  &\; -1 \\
1 & 0 \end{array}\right).\]
So the unique class $C$ in Corollary~\ref{correg1} is the $\bG^F$-conjugacy 
class of~$g$. Now $\Sigma^F$ splits into two classes in $\bG^F$, with 
representatives $g$ and 
\[\renewcommand{\arraystretch}{0.8} g':=\left(\begin{array}{cc} 
\; 1 \; & \; 0\;  \\ \xi & 1 \end{array}\right)\in\Sigma^F, \qquad 
\mbox{where $\xi\in \F_q^\times$ is not a square in $\F_q^\times$}.\]
(One checks that, indeed, $g'\in \bB^F\dot{s}_1\bB^F$ but 
$g'\not\in \bU^F\dot{s}_1\bU^F$.) Furthermore, one checks that 
$g$ and $g^{-1}$ are conjugate in $\bG^F$ if and only if $-1$ is a
square in $\F_q^\times$, that is, if and only if $q\equiv 1 \bmod 4$.
Hence, if $q\equiv 3 \bmod 4$, then $C\neq C^{-1}$. 
\end{exmp}
 
\begin{rem} \label{hecke} Assume that $\bG$ is of split type (then
$\sigma=\mbox{id}_\bW$). Let $\Sigma$ be an arbitrary $F$-stable 
conjugacy class of $\bG$ and $w\in \bW$. For any $g\in\Sigma^F$ 
denote by $C_g$ the $\bG^F$-conjugacy class of~$g$. Then the 
cardinalities $|C_g\cap \bB^F\dot{w}\bB^F|$ can be computed using the 
representation theory of $\bG^F$; see, e.g., \cite[1.2(a)]{Lfrom}. For 
this purpose, we consider the induced character $\Ind_{\bB^F}^{\bG^F}(1)$
(where $1$ stands for the trivial character of $\bB^F$) and let
$\cH_q$ be the corresponding Hecke algebra, that is, the endomorphism
algebra of a $\K \bG^F$-module affording $\Ind_{\bB^F}^{\bG^F}(1)$.
This algebra has a standard basis usually denoted 
by $\{T_w\mid w\in \bW\}$, where 
\[T_{w_\alpha}^2=qT_1+(q-1)T_{w_\alpha}\qquad\mbox{for every simple 
root~$\alpha\in\Phi$}.\]
There is a bijection, $\phi \leftrightarrow \phi^{(q)}$, between 
$\Irr(\bW)$ and the irreducible characters of~$\cH_q$ (which is 
canonical once a square root $\sqrt{q}\in\K$ has been fixed; see, 
e.g., \cite[\S 9.3]{gepf}.) Via this correspondence, the irreducible 
characters of $\bG^F$ that occur in $\Ind_{\bB^F}^{\bG^F}(1)$ are 
parametrised by $\Irr(\bW)$ (see, e.g., \cite[\S 8.4]{gepf}); we denote 
by $[\phi]\in\Irr(\bG^F)$ the character corresponding to $\phi\in
\Irr(\bW)$. Then we have:
\begin{equation*}
|C_g \cap \bB^F\dot{w}\bB^F|= \frac{|\bB^F|}{|C_\bG(g)^F|} 
\sum_{\phi\in\Irr(\bW)} \phi^{(q)}(T_w)\,[\phi](g) \tag{a}
\end{equation*}
for any $w\in\bW$ and any $g\in\bG^F$. (See \cite[1.2(a)]{Lfrom} and 
\cite[Remark~3.6]{carp} for further details and references.) The 
values $\phi^{(q)}(T_w)$ are explicitly known (or there are explicit
combinatorial algorithms); see \cite{gepf}. Hence, if we have sufficient 
information on the values of $[\phi]\in\Irr(\bG^F)$, then we can work 
out the cardinality $|C_g \cap \bB^F\dot{w}\bB^F|$, and this will be 
useful to identify $C_g\subseteq \Sigma^F$.

Assume now that we are in the setting of \S \ref{remreg}, where 
$\bG$ is simple and simply connected, $\Sigma$ consists of regular 
elements and $w=w_c$ is a Coxeter element. Also assume that the natural 
map $Z(\bG)\rightarrow A_\bG(g_1)$ is injective and $F$ acts trivially
on $A_\bG(g_1)$. Let $C_1,\ldots,C_r$ be the $\bG^F$-conjugacy classes
that are contained in $\Sigma^F$ and have a non-empty intersection 
with $\bB^F\dot{w}_c\bB^F$. Then, by Example~\ref{correg2}, we 
have $r=|Z(\bG)|$ and each set $C_i\cap \bB^F\dot{w}_c\bB^F$ is 
a single $\bB^F$-orbit, of size $|\bB^F|/r$. Hence, the above
identity (a) can be expressed as follows.
\begin{equation*}
\sum_{\phi\in\Irr(\bW)} \phi^{(q)}(T_w)\,[\phi](g)=\left\{
\begin{array}{cl} \frac{1}{r}|C_\bG(g)^F| & \quad\mbox{if $g \in C_1
\cup\ldots \cup C_r$},\\ 0 & \quad\mbox{if $g\in\Sigma^F\setminus 
(C_1\cup \ldots \cup C_r)$}.\end{array}\right.
\tag{b}
\end{equation*}
This identity can be exploited to obtain information on the character 
values $[\phi](g)$ and, hence, potentially, on the unknown scalars 
$\zeta$ in \S \ref{ualm}($\clubsuit^\prime$); see the proof 
of Proposition~\ref{hecke7} below for an example. 
\end{rem}

%%%%%%%%%%%%%%%%%%%%%%%%%%%%%%%%%%%%%%%%%%%%%%%%%%%%%%%%%%%%%%%%%%%%%%%%%%%
\section{Cuspidal unipotent character sheaves in type $E_6$} \label{sece6}

Throughout this section, let $\bG$ be simple, simply connected of type 
$E_6$. Let $q=p^f$ (where $f\geq 1$) be such that $F\colon \bG 
\rightarrow \bG$ defines an $\F_q$-rational structure. Except for 
\S \ref{e6t} (at the very end), we assume that $\bG$ is of split
type; thus, $\sigma=\mbox{id}_\bW$ and the permutation $\alpha\mapsto
\alpha^\dagger$ of $\Phi$ is the identity. Let $\Delta=\{\alpha_1,\alpha_2,
\alpha_3,\alpha_4,\alpha_5,\alpha_6\}$ be the set of simple roots in 
$\Phi^+$, where the labelling is chosen as follows.
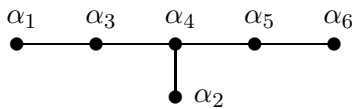
\begin{center}
\begin{picture}(270,37)
\put(132, 0){$\alpha_2$}
\put( 61,30){$\alpha_1$}
\put( 91,30){$\alpha_3$}
\put(121,30){$\alpha_4$}
\put(151,30){$\alpha_5$}
\put(181,30){$\alpha_6$}
\put(125, 2){\circle*{5}}
\put( 65,22){\circle*{5}}
\put( 95,22){\circle*{5}}
\put(125,22){\circle*{5}}
\put(155,22){\circle*{5}}
\put(185,22){\circle*{5}}
\put(125,22){\line(0,-1){20}}
\put( 65,22){\line(1,0){30}}
\put( 95,22){\line(1,0){30}}
\put(125,22){\line(1,0){30}}
\put(155,22){\line(1,0){30}}
\end{picture}
\end{center}
If $p=3$, then the cuspidal character sheaves and almost characters have
been considered by Hetz \cite{Het1}. So assume from now on that $p\neq 3$. 
Let $\alpha_0\in\Phi$ be the unique root of maximal height and consider the
subsystem $\Phi_0\subseteq \Phi$ of type $A_2\times A_2\times A_2$ spanned 
by $\{\alpha_1,\alpha_2,\alpha_3,\alpha_5,\alpha_6,\alpha_0\}$. 
The relevance of this particular example is that $\Phi_0$ occurs in the 
classification of cuspidal unipotent character sheaves on $\bG$; see 
\cite[Prop.~20.3]{L2d} (and also the remarks in \cite[5.2]{S3}). Using
{\sf CHEVIE}, we find that the 
unique set $\Delta_0$ of simple roots in $\Phi_0\cap\Phi^+$ is given by
\[ \Delta_0=\{\; \alpha_1,\;\alpha_2,\;\alpha_3,\;\alpha_5, \;
\alpha_6,\quad \alpha_0':=\alpha_1{+}\alpha_2 {+}2\alpha_3{+}
3\alpha_4{+}2\alpha_5{+}\alpha_6\;\}.\]
Furthermore, there are three equivalence classes of pairs $(\Phi',w)
\in\Xi$ under $\sim$, where $\Phi'=\Phi_0$; representatives 
$(\Phi_0,d_i)$, where $d_i\in\bW$ has minimal length in $\bW(\Phi_0)
d_i$ for $i=1,2,3$, are given as follows.
\[\begin{array}{l@{\hspace{15pt}}c@{\hspace{15pt}}c} \hline 
d_i & \mbox{permutation} & \mbox{$\sigma_i'$-classes} \\\hline d_1=
1_\bW & () & 27 \\ d_2=431543654 & (1,4)(2,6)(3,5) & 9 \\ 
d_3= 425431654234 & (1,5,2)(3,4,6) & 3 \\\hline
\end{array}\]
Here, e.g., $4315\cdots$ means the product $w_{\alpha_4}w_{\alpha_3}
w_{\alpha_1} w_{\alpha_5}\cdots$ in $\bW$. Otherwise, the conventions 
are the same as in Example~\ref{exg2a}. In particular, recall that the 
permutation in the second column refers to the simple roots in $\Delta_0$,
as listed above, not to those in $\Delta$. (So, e.g., the cycle $(3,4,6)$ 
means that $\alpha_3\mapsto \alpha_5\mapsto \alpha_0'\mapsto \alpha_3$.) 

\begin{abs} {\bf The subgroup $\bH'=\langle \bT_0, \bU_\alpha\;(\alpha\in
\Phi_0)\rangle$.} \label{he6} For each root $\alpha\in\Phi$, denote by 
$\alpha^\vee \colon k^\times \rightarrow\bT_0$ the corresponding coroot. 
Since $\bG$ is simply connected, every $t\in\bT_0$ has a unique expression 
$t=h(\xi_1,\ldots,\xi_6):=\prod_{1\leq i\leq 6} \alpha_i^\vee(\xi_i)$ where
$\xi_i \in k^\times$ for $1\leq i\leq 6$. By \cite[Example~1.5.6]{gema}, we 
have 
\[Z(\bG)=\{h(\xi,1,\xi^{-1},1,\xi, \xi^{-1}) \mid \xi\in k^\times, 
\xi^3=1\}.\]
A similar computation shows that 
\[ Z(\bH')=\{h(\xi,1,\xi^{-1},1,\zeta,\zeta^{-1}) \mid \xi,\zeta\in 
k^\times, \xi^3=\zeta^3=1\}.\]
Thus, $Z(\bH')$ is generated by $Z(\bG)$ and any fixed $t\in Z(\bH')
\setminus Z(\bG)$. Since $q$ is not a power of $3$, we have 
$|Z(\bG)|=3$ and $|Z(\bH')|=9$. Given $t=h(\xi,1,\xi^{-1},1,\zeta,
\zeta^{-1})\in Z(\bH')$ (where $\xi^3=\zeta^3=1$), we have $C_\bG(t)=\bH'$ 
if and only if $\xi\neq \zeta$. Furthermore, one easily checks that 
\begin{align*}
\dot{d}_2^{-1}t\dot{d}_2&=z_2(t)t^{-1}\quad\mbox{where}\quad 
z_2(t):=h\bigl(\xi\zeta,1,(\xi\zeta)^{-1},1,\xi\zeta,(\xi\zeta)^{-1}\bigr)
\in Z(\bG),\\ \dot{d}_3^{-1}t\dot{d}_3&=z_3(t)t \quad\quad\mbox{where}\quad 
z_3(t):=h\bigl(\xi\zeta^{-1},1,\xi^{-1}\zeta,1,\xi\zeta^{-1},\xi^{-1}\zeta
\bigr)\in Z(\bG).
\end{align*}
These two relations show that all elements in $Z(\bH')\setminus
Z(\bG)$ are conjugate in $\bG$. In particular, if we choose $\zeta=
\xi^{-1}\neq 1$, then $\dot{d}_2^{-1}t \dot{d}_2=t^{-1}$. Thus, if 
$\cC$ denotes the $\bG$-conjugacy class of the elements in $Z(\bH')
\setminus Z(\bG)$, then 
\[ F(\cC)=\cC,\qquad \cC=\cC^{-1} \qquad\mbox{and}\qquad 
\cC=Z(\bG)\cC.\]
In order to see that $\cC$ is $F$-stable, we argue as follows. If 
$q\equiv 1\bmod 3$, then $F(z)=z$ for all $z\in Z(\bH')$. On the other
hand, if $q\equiv 2 \bmod 3$, then $F(z)=z^{-1}$ for all $z\in Z(\bH')$ 
and, hence, $Z(\bH')^F =\{1\}$. But, if $t=h(\xi,1,\xi^{-1},1,\zeta,
\zeta^{-1}) \in \cC$ with $\zeta=\xi^{-1}\neq 1$ as above, then 
$F(t)=t^{-1}=\dot{d}_2^{-1}t\dot{d}_2 \in \cC$, as required.
\end{abs}

\begin{abs} {\bf The conjugacy class $\Sigma\subseteq \bG$.} \label{se6}
Let us fix an element $s_1\in \cC$; since $\cC$ is $F$-stable, we 
may assume that $F(s_1)=s_1$. Now $\bH_1':=C_\bG(s_1)$ is conjugate to
$\bH'$ in~$\bG$. Let $u_1 \in \bH_1'$ be regular unipotent; since all 
regular unipotent elements in $\bH_1'$ are conjugate in
$\bH_1'$, we may assume that $F(u_1)=u_1$. Let $\Sigma$ be the 
$\bG$-conjugacy class of $g_1:=s_1u_1$; then $\Sigma$ is $F$-stable since
$F(g_1)=g_1$. Since $Z(\bG)\cC=\cC=\cC^{-1}$, one also deduces that 
$Z(\bG)\Sigma=\Sigma=\Sigma^{-1}$. We claim that 
\[A_\bG(g_1) \mbox{ is generated by $\bar{s}_1$ and all $\bar{z}$ 
for $z\in Z(\bG)$};\]
here, for any $c\in C_\bG(g_1)$, we denote by $\bar{c}$ the image of 
$c$ in $A_\bG(g_1)$. Indeed, we have $C_\bG(g_1)=C_{\bH_1'}(u_1)$ 
and so $C_\bG^\circ(g_1)=C_{\bH_1'}^\circ(u_1)$. Since we are in good 
characteristic (inside $\bH_1'\cong \bH'$, which is of type $A_2{\times} 
A_2 {\times} A_2$) and $Z(\bH_1')$ is finite, it follows that $A_\bG(g_1)
=A_{\bH_1'}(u_1)\cong Z(\bH_1')$, where the isomorphism is induced by the 
natural map $Z(\bH_1')\subseteq C_{\bH_1'}(u_1) \rightarrow A_{\bH_1'}
(u_1)$; see \cite[\S 12.3]{rDiMi}. Note also that $\bar{g}_1=\bar{s}_1$.
By \S \ref{he6}, we have that $Z(\bH')$ is generated by $Z(\bG)$ 
and any fixed element in $Z(\bH')\setminus Z(\bG)$. Consequently,
$Z(\bH_1')$ is generated by $Z(\bG)$ and~$s_1$, which implies the 
above claim.

Note that $F$ acts trivially on $A_\bG(g_1)$ if $q\equiv 1\bmod 3$. 
On the other hand, if $q\equiv 2 \bmod 3$, then $F(z)=z^{-1}$ for all
$z\in Z(\bG)$ and, hence, $F$ acts non-trivially on $A_\bG(g_1)$;
in this case, we have $A_\bG(g_1)^F=\langle \bar{s}_1\rangle\cong \Z/3\Z$.
Hence, the set $\Sigma^F$ splits into an odd number (either $9$ or $3$) of
conjugacy classes in $\bG^F$. So, among these classes, there must be at 
least one that is equal to its inverse; we now choose $g_1\in\Sigma^F$ 
to be in such a class; thus, $g_1$ is conjugate to $g_1^{-1}$ in $\bG^F$ 
(and not just in~$\bG$). Note also that, using Lemma~\ref{lemreg3}(b), we 
could fix the $\bG^F$-conjugacy class of $g_1$ completely, by requiring 
that $g_1\in C$, where $C$ is the $\bG^F$-conjugacy class determined by 
$\Sigma$ (and the choice of $\dot{w}_c$) as in Corollary~\ref{correg1}.
\end{abs}

\begin{abs} {\bf Cuspidal unipotent character sheaves.} \label{ce6}
First we consider the group $\tilde{\bG}:=\bG/Z(\bG)$. Let $\pi\colon 
\bG\rightarrow \tilde{\bG}$ be the canonical map; let $\tilde{\Sigma}
:=\pi(\Sigma)$ and $\tilde{g}_1:=\pi(g_1)$. By the proof of 
\cite[Prop.~20.3(a)]{L2d} (see also \cite[4.6, 5.2]{S3}), there are two 
cuspidal unipotent character sheaves $\tilde{A}_1$ and $\tilde{A}_2$
of $\tilde{\bG}$; they have support $\tilde{\Sigma}$ and they are 
$F$-invariant. As explained in the proof of \cite[Cor.~20.4]{L2d} (see 
also \cite[p.~347]{S3}), the local systems (see \S \ref{charf})
associated with $\tilde{A}_1$ and $\tilde{A}_2$ are one-dimensional; 
they correspond to linear characters $\tilde{\psi}_1$ and $\tilde{\psi}_2$ 
of $A_{\tilde{\bG}}(\tilde{g}_1)$, such that $\tilde{\psi}_1
(\bar{\tilde{g}}_1)=\theta$ and $\tilde{\psi}_2(\bar{\tilde{g}}_1)=
\theta^2$, where $1\neq \theta\in\K^\times$ is a fixed third root of 
unity. (Here, $\bar{\tilde{g}}_1$ denotes the image of $\tilde{g}_1$ in 
$A_{\tilde{\bG}}(\tilde{g}_1)$.) Now use $\pi$ to go back to $\bG$. First
note that $\pi$ canonically induces a group homomorphism $\bar{\pi}
\colon A_\bG(g_1) \rightarrow A_{\tilde{\bG}}(\tilde{g}_1)$. Hence, we
obtain irreducible characters of $A_\bG(g_1)$ by setting 
\[\psi_1:=\tilde{\psi}_1\circ \bar{\pi}\in\Irr(A_\bG(g_1))\qquad\mbox{and}
\qquad \psi_2:=\tilde{\psi}_2\circ \bar{\pi}\in\Irr(A_\bG(g_1)).\]
By the proof of \cite[Prop.~20.3(b)]{L2d}, $A_1:=\pi^*\tilde{A}_1$ and 
$A_2:=\pi^*\tilde{A}_2$ are cuspidal unipotent character sheaves on $\bG$ 
(and these are the only ones); they have support $\Sigma$ and they 
correspond to the linear characters $\psi_1$ and $\psi_2$ of $A_\bG(g_1)$.
(See the general reduction techniques described in \cite[2.10]{LuIC}.) 
Finally note that, clearly, the image of $Z(\bG)$ in $A_\bG(g_1)$ is
contained in $\ker(\bar{\pi})$. Hence, we have
\[ \psi_1(\bar{s}_1)=\theta,\qquad \psi_2(\bar{s}_1)=
\theta^2, \qquad \psi_1(\bar{z})=\psi_2(\bar{z})=1\quad
\mbox{for $z\in Z(\bH_1')^F$}.\] 
(Recall that $\bar{g}_1=\bar{s}_1$.) Thus, $\psi_1$ and $\psi_2$ are 
completely determined, where $\psi_2$ is the complex conjugate of~$\psi_1$. 
Furthermore, the roots of unity attached to $A_1$ and $A_2$ as in 
\S \ref{charf} are $\lambda_{A_1}=\theta$ and $\lambda_{A_2}=
\theta^2$. Using $\psi_1$ and $\psi_2$, we can now write down 
characteristic functions of $A_1$ and $A_2$, as in \S \ref{charf};
we have $\chi_{g_1,\psi_i}=\chi_{\tilde{g}_1,\tilde{\psi}_i}\circ \pi$ 
for $i=1,2$. (Recall that $\Sigma=Z(\bG)\Sigma$ and so $\Sigma=
\pi^{-1}(\tilde{\Sigma})$.)
\end{abs}

\begin{abs} {\bf Unipotent characters and almost characters.} \label{ue6}
Let again $\tilde{\bG}=\bG/Z(\bG)$ and $\pi\colon \bG\rightarrow 
\tilde{\bG}$ as above. First note that the unipotent characters of 
$\bG^F$ and $\tilde{\bG}^F$ can be canonically identified via~$\pi$ 
(see, e.g., \cite[Prop.~2.3.15]{gema}). They are parametrised by a 
certain et $X(\bW)$ (which only depends on $\bW$); we use the notation 
in the table on \cite[p.~363]{LuB}. The unipotent almost characters are
also parametrised by $X(\bW)$. The interesting cases for us are as follows.
\begin{align*}
R_{(g_3,\theta)}&:= \textstyle\frac{1}{3}\bigl([80_s]+[20_s]
-[10_s]-[90_s]+2E_6[\theta]-E_6[\theta^2]\bigr),\\
R_{(g_3,\theta^2)}&:=\textstyle\frac{1}{3}\bigl([80_s]+
[20_s] -[10_s]-[90_s]-E_6[\theta]+2E_6[\theta^2]\bigr).
\end{align*}
Here, $80_s$, $20_s$ etc.\ are irreducible characters of $\bW$
(denoted as in \cite[Chap.~4]{LuB}); then $[80_s]$, $[20_s]$ etc.\ are 
the corresponding irreducible constituents of 
$\mbox{Ind}_{\bB^F}^{\bG^F}(1)$ where $1$ stands for the trivial 
character of $\bB^F$; the characters $E_6[\theta]$ and $E_6[\theta^2]$ 
are cuspidal unipotent. (Note also that the ``$g_3$'' in $(g_3,\theta)$ 
and $(g_3,\theta^2)$ has nothing to do with elements in $\bG^F$; these
are just notations for parameters in $X(\bW)$.) Now consider the
two cuspidal unipotent character sheaves $A_1$ and $A_2$ described above, 
with characteristic functions $\chi_{g_1,\psi_1}$ and $\chi_{g_1,\psi_2}$. 
By the main result of \cite[\S 4]{S3} (see also \cite[5.2]{S3}), there 
are scalars $\zeta, \zeta'\in\K$ of absolute value~$1$ such that
\[R_{(g_3,\theta)}=\zeta\chi_{g_1,\psi_1} \qquad \mbox{and}\qquad 
R_{(g_3,\theta^2)}=\zeta' \chi_{g_1,\psi_2}.\]
(More precisely, in \cite[\S 4]{S3}, this is proved for $\tilde{\bG}$ 
but the discussion in \S \ref{ce6} shows that this also holds
for $\bG$, with $\psi_1$ and $\psi_2$ as above.)
By \cite[Table~1]{my03}, the characters $E_6[\theta]$ and
$E_6[\theta^2]$ are complex conjugate to each other, and their values
lie in the field $\Q(\theta)$; furthermore, all characters $[\phi]$ 
(where $\phi\in\Irr(\bW)$) are rational-valued; see \cite[Prop.~5.6]{my03}.
We conclude that the class functions $R_{(g_3,\theta)}$ and $R_{(g_3,
\theta^2)}$ are complex conjugate to each other, and their values lie 
in $\Q(\theta)$. Now, since $\dim \bG-\dim \Sigma=\dim \bT_0=6$, we have 
\begin{align*}
R_{(g_3,\theta)}(g_1)&=\zeta\chi_{g_1,\psi_1}(g_1)=\zeta q^3,\\
R_{(g_3,\theta^2)}(g_1)&=\zeta'\chi_{g_1,\psi_2}(g_1)=\zeta' q^3.
\end{align*}
Thus, we can already conclude that $\zeta'=\overline{\zeta}$.
Since $g_1$ is conjugate to $g_1^{-1}$ in $\bG^F$, we have
$E_6[\theta](g_1)=E_6[\theta](g_1^{-1})=\overline{E_6[\theta]}(g_1)=
E_6[\theta^2](g_1)$. Consequently, we also have $R_{(g_3,\theta)}(g_1)=
\overline{R_{(g_3,\theta)}(g_1)}$ and so $\zeta=\overline{\zeta}$. Hence, 
we must have $\zeta=\zeta'=\pm 1$, since $\zeta$ has absolute value~$1$.
\end{abs}

\begin{prop} \label{e6zeta} In the above setting, recall that $g_1\in 
\Sigma^F$ is conjugate to $g_1^{-1}$ in $\bG^F$. Then we have $\zeta=
\zeta'=1$, that is, $R_{(g_3,\theta)}=\chi_{g_1,\psi_1}$ and $R_{(g_3,
\theta^2)}= \chi_{g_1,\psi_2}$.
\end{prop}

\begin{proof} Inverting the matrix relating unipotent 
characters and unipotent almost characters, we obtain:
\[ E_6[\theta] =\textstyle\frac{1}{3}\bigl(R_{80_s}+R_{20_s}
-R_{10_s}-R_{90_s}+2R_{(g_3,\theta)}-R_{(g_3,\theta^2)}\bigr).\]
Using the formula for $R_\phi$ in \S \ref{ualm}, the (known) 
character table of $\bW$ and Example~\ref{reguni}, we find that 
\[ R_{80_s}(g_1)=R_{20_s}(g_1)=R_{90_s}(g_1)=0 \qquad\mbox{and}\qquad
R_{10_s}(g_1)=\epsilon,\]
where $\epsilon=\pm 1$ is such that $q\equiv \epsilon \bmod 3$.
This yields $E_6[\theta](g_1)=\textstyle\frac{1}{3}(-\epsilon+
\zeta q^3)\in\Q$, where the left hand side is an algebraic integer.
Hence, $3$ must divide $\zeta q^3-\epsilon\in\Z$. Since $\zeta=\pm 1$,
the only possibility is that $\zeta=1$.  
\end{proof}

\begin{table}[htbp] \caption{Values of $E_6[\theta]$ on $\Sigma^F$} 
\label{vale6a}
\begin{center}
$\renewcommand{\arraystretch}{1.3} \renewcommand{\arraycolsep}{10pt} 
\begin{array}{|c|c|c|c|} 
\hline q\equiv 1 \bmod 3 & \{\bar{1},\;\bar{z},\; \bar{z}^2\} & \{\bar{s}_1,
\;\bar{s}_1\bar{z},\; \bar{s}_1\bar{z}^2\} & \{\bar{s}_1^2,\;
\bar{s}_1^2\bar{z},\; \bar{s}_1^2 \bar{z}^2\}\\
\hline E_6[\theta] & \frac{1}{3}(q^3-1) & \frac{1}{3}(q^3-1)+q^3\theta & 
\frac{1}{3}(q^3-1)+q^3\theta^2 \\ \hline 
\end{array}$\\[5pt]
$\renewcommand{\arraystretch}{1.3} \renewcommand{\arraycolsep}{4pt} 
\begin{array}{|c|c|c|c|} 
\hline \;\; q\equiv 2 \bmod 3\;\;  & \bar{1} & \bar{s}_1 & \bar{s}_1^2 \\
\hline E_6[\theta] & \quad\frac{1}{3}(q^3+1)\quad & 
\;\frac{1}{3}(q^3+1)+q^3\theta \;&\;\frac{1}{3}(q^3+1)+q^3 \theta^2 \;\\
\hline \multicolumn{4}{l}{\text{(Here, $g_1=s_1u_1\in\Sigma^F$
is such that $g_1$ and $g_1^{-1}$ are conjugate in $\bG^F$.)}}
\end{array}$
\end{center}
\end{table}

The resulting values of $E_6[\theta]$ on the conjugacy classes of
$\bG^F$ that are contained in $\Sigma^F$ are displayed in 
Table~\ref{vale6a}, where $z$ denotes a non-trivial element in
$Z(\bG)$ when $q\equiv 1 \bmod 3$. (Recall that $\Sigma^F$ splits into 
$9$ classes if $q\equiv 1 \bmod 3$, and into $3$ classes if $q\equiv 2 
\bmod 3$; these classes are parametrised by representatives of the 
$F$-conjugacy classes of $A_\bG(g_1) \cong Z(\bH_1')$.) 

\begin{abs} {\bf Twisted type.} \label{e6t} We keep the above notation, 
but now assume that $(\bG,F)$ is non-split. Then the induced automorphism
$\sigma \colon \bW \rightarrow \bW$ is given by conjugation with the 
longest element $w_0\in\bW$. The permutation $\alpha\mapsto 
\alpha^\dagger$ of $\Phi$ is of order~$2$, such that $\alpha_1^\dagger=
\alpha_6$, $\alpha_3^\dagger=\alpha_5$, $\alpha_2^\dagger=\alpha_2$ and 
$\alpha_4^\dagger=\alpha_4$. The two cuspidal unipotent character sheaves 
$A_1$ and $A_2$ considered above are also $F$-stable; see
\cite[Cor.~20.4]{L2d} and its proof. In all essential points, we can
further argue as above, so we just state the main results. To begin with, 
the subsystem $\Phi_0\subseteq \Phi$ is invariant under~$\dagger$. Again, 
there are three equivalence classes of pairs $(\Phi',w)\in \Xi$ under 
$\sim$, where $\Phi'=\Phi_0$; representatives $(\Phi_0,d_i)$, where 
$d_i \in \bW$ has minimal length in $\bW(\Phi_0)d_i$ for $i=1,2,3$, are 
given as follows.
\[\begin{array}{l@{\hspace{15pt}}c@{\hspace{15pt}}c} \hline 
d_i & \mbox{permutation} & \mbox{$\sigma_i'$-classes} \\\hline 
d_1=1_\bW & (1,5)(3,4) & 9 \\
d_2=431543654 & (1,3)(2,6)(4,5) & 27 \\ 
d_3'=423143542314354 & (1,6,5,3,2,4) & 3 \\ \hline
\end{array}\]
Given $t=h(\xi_1,\ldots,\xi_6)\in \bT_0$, with $\xi\in k^\times$ for
all $i$, we now have 
\[ F(t)=h(\xi_6^q,\xi_2^q,\xi_5^q,\xi_4^q,\xi_3^q,\xi_1^q) \qquad 
\mbox{for all $\xi \in k^\times$}.\]
Recall that $Z(\bH')=\{h(\xi,1,\xi^{-1},1,\zeta,\zeta^{-1})\mid \xi,
\zeta \in k^\times, \xi^3=\zeta^3=1\}$; also recall the definition
of the $\bG$-conjugacy class $\cC$ from \S \ref{he6}.

Let $t=h(\xi,1,\xi^{-1},1,\zeta,\zeta^{-1})\in Z(\bH')$. Assume first
that $q\equiv 1\bmod 3$. Then $F(t)=t$ if and only if $\xi=\zeta^{-1}$. 
Thus, there exists some $t\in Z(\bH')^F$ such that $C_\bG(t)=\bH'$. On 
the other hand, if $q\equiv 2 \bmod 3$, then one checks that $F(t)=
\dot{d}_2^{-1}t \dot{d}_2$ for all $t\in Z(\bH')$. In particular, there 
exists some $t\in Z(\bH')$ such that $C_\bG(t)=\bH'$ and $F(t)=
\dot{d}_2^{-1}t \dot{d}_2$. Hence, in both cases, the class $\cC$ is 
$F$-stable. We now define $\Sigma$ as in \S \ref{se6}; let $g_1
\in \Sigma^F$.  It follows again that 
\[ A_\bG(g_1) \mbox{ is generated by $\bar{s}_1$ and all $\bar{z}$ for
$z\in Z(\bG)$}.\]
We obtain characteristic functions $\chi_{g_1,\psi_1}$ and
$\chi_{g_1,\psi_2}$ for $A_1$ and $A_2$, respectively, by exactly
the same formulae as in \S \ref{ce6}. By the main results
of Shoji \cite[\S 4]{S3} (see also \cite[4.8, 5.2]{S3}), there are
scalars $\zeta,\zeta'\in\K$ of absolute value~$1$ such that
$\tilde{R}_{(g_3,\theta)}=\zeta \chi_{g_1,\psi_1}$ and $\tilde{R}_{(g_3,
\theta^2)}=\zeta' \chi_{g_1,\psi_2}$. Now note that, since $\bG,F$ is
not of split type, the almost characters are only well-defined up to a 
sign. In accordance with \cite[4.19]{LuB}, we define the almost
characters $\tilde{R}_{(g_3,\theta)}$ and $\tilde{R}_{(g_3,\theta^2)}$ 
as follows:
%the following linear combinations of unipotent characters:
\begin{align*}
\tilde{R}_{(g_3,\theta)}&:= \textstyle\frac{1}{3}\bigl(
{^2\!E}_6[1] +[\phi_{12,4}] -[\phi_{6,6}'] -[\phi_{6,6}'']
+2{\cdot}{^2\!E}_6[\theta]- {^2\!E}_6[\theta^2]\bigr),\\
\tilde{R}_{(g_3,\theta^2)}&:= \textstyle\frac{1}{3}\bigl(
{^2\!E}_6[1] +[\phi_{12,4}] -[\phi_{6,6}'] -[\phi_{6,6}'']
-{^2\!E}_6[\theta]+2{\cdot}{^2\!E}_6[\theta^2]\bigr).
\end{align*}
Here, we use the notation in \cite[p.~481]{Ca2}. Thus, $\phi_{12,4}$, 
$\phi_{6,6}'$ etc.\ are irreducible characters of $\bW^\sigma:=\{w\in
\bW\mid \sigma(w)=w\}$ (which is a Weyl group of type $F_4$); then 
$[\phi_{12,4}]$, $[\phi_{6,6}']$ etc.\ are the corresponding irreducible 
constituents of $\mbox{Ind}_{\bB^F}^{\bG^F}(1)$; characters denoted like
${^2\!E}_6[1]$ are cuspidal unipotent. (We refer to \cite{Ca2} 
instead of \cite{LuB}, because the table of unipotent characters for 
twisted $E_6$ is not explicitly printed in \cite{LuB}.) By the same
argument as in \S \ref{ue6}, we can choose $g_1\in \Sigma^F$ such that 
$g_1$ is conjugate to $g_1^{-1}$ in $\bG^F$. Hence, as before, we 
already know that $\zeta'=\overline{\zeta}=\zeta=\pm 1$. So it only 
remains to decide whether $\zeta$ equals $1$ or $-1$. 
\end{abs}

\begin{prop} \label{te6zeta} Recall that $g_1\in \Sigma^F$ is conjugate
to $g_1^{-1}$ in $\bG^F$. With $\tilde{R}_{(g_3,\theta)}$ and
$\tilde{R}_{(g_3,\theta^2)}$ defined as above, we have $\zeta=
\zeta'=1$.
\end{prop}

\begin{proof} In accordance with \cite[4.19]{LuB}, for all $\phi\in
\Irr(\bW)$ occurring in  the above expressions for 
$\tilde{R}_{(g_3,\theta)}$ and $\tilde{R}_{(g_3,\theta^2)}$, we now define 
\[ \tilde{R}_\phi:=-\frac{1}{|\bW|} \sum_{w\in \bW} \phi(ww_0)
R_{\bT_w}^\bG(1).\]
(This corresponds to the ``preferred extensions'' in \cite[17.2]{L2d};
note that the $a$-invariants of all $\phi$ as above are~$3$, which 
accounts for the minus sign in the definition of $\tilde{R}_\phi$.) 
Then, by \cite[Main Therem~4.23]{LuB}, we have:
\[ {^2\!E}_6[\theta] =\textstyle\frac{1}{3}\bigl(
\tilde{R}_{80_s}+\tilde{R}_{20_s}-\tilde{R}_{10_s}-\tilde{R}_{90_s}+
2\tilde{R}_{(g_3,\theta)} -\tilde{R}_{(g_3,\theta^2)}\bigr).\]
Using the formula in Example~\ref{reguni}, we obtain 
\[ \tilde{R}_{80_s}(g_1)=\tilde{R}_{20_s}(g_1)=\tilde{R}_{90_s}(g_1)=0
\qquad\mbox{and}\qquad \tilde{R}_{10_s}(g_1)=\varepsilon,\]
where $\varepsilon=\pm 1$ is such that $q\equiv \varepsilon\bmod 3$. 
This yields the relation ${^2\!E}_6[\theta](g_1)=\frac{1}{3}
(\zeta q^3-\varepsilon)$, which implies that $\zeta=1$ regardless of
whether $\varepsilon$ is $+1$ or $-1$. 
\end{proof}

%%%%%%%%%%%%%%%%%%%%%%%%%%%%%%%%%%%%%%%%%%%%%%%%%%%%%%%%%%%%%%%%%%%%%%%%%%%
\section{Cuspidal unipotent character sheaves in type $E_7$} \label{sece7}

Throughout this section, let $\bG$ be simple, simply connected of type 
$E_7$. Let $q=p^f$ (where $f\geq 1$) be such that $F\colon \bG 
\rightarrow \bG$ defines an $\F_q$-rational structure. Here, $\bG$ is of 
split type; thus, $\sigma=\mbox{id}_\bW$ and the permutation $\alpha
\mapsto \alpha^\dagger$ of $\Phi$ is the identity. Let $\Delta=\{\alpha_1,
\alpha_2,\alpha_3,\alpha_4,\alpha_5,\alpha_6,\alpha_7\}$ be the set of 
simple roots in $\Phi^+$, where the labelling is chosen as follows.
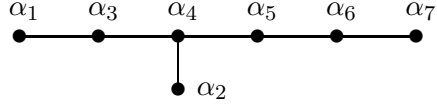
\begin{center}
\begin{picture}(270,35)
\put(132, 0){$\alpha_2$}
\put( 61,30){$\alpha_1$}
\put( 91,30){$\alpha_3$}
\put(121,30){$\alpha_4$}
\put(151,30){$\alpha_5$}
\put(181,30){$\alpha_6$}
\put(211,30){$\alpha_7$}
\put(125, 2){\circle*{5}}
\put( 65,22){\circle*{5}}
\put( 95,22){\circle*{5}}
\put(125,22){\circle*{5}}
\put(155,22){\circle*{5}}
\put(185,22){\circle*{5}}
\put(215,22){\circle*{5}}
\put(125,22){\line(0,-1){20}}
\put( 65,22){\line(1,0){30}}
\put( 95,22){\line(1,0){30}}
\put(125,22){\line(1,0){30}}
\put(155,22){\line(1,0){30}}
\put(185,22){\line(1,0){30}}
\end{picture}
\end{center}
If $p=2$, then the cuspidal character sheaves and almost characters have
been considered by Hetz \cite{Het2}. So assume from now on that $p\neq 2$. 
Let $\alpha_0\in\Phi$ be the unique root of maximal height and consider the
subsystem $\Phi_0\subseteq \Phi$ of type $A_3\times A_3\times A_1$
spanned by $\{\alpha_1,\alpha_2,\alpha_3,\alpha_5,\alpha_6,\alpha_7,
\alpha_0\}$. Again, the relevance of this particular example is 
that $\Phi_0$ occurs in the classification of cuspidal character sheaves 
on $\bG$; see \cite[Prop.~20.3]{L2d} (and also \cite[5.2]{S3}). Using
{\sf CHEVIE}, we find that the unique 
set $\Delta_0$ of simple roots in $\Phi_0\cap\Phi^+$ is given by 
\[\Delta_0=\{\; \alpha_1,\;\alpha_2,\;\alpha_3,\;\alpha_5,\;\alpha_6,\; 
\alpha_7,\;\alpha_0':=\alpha_1{+}2\alpha_2 {+}2\alpha_3{+}4\alpha_4{+}
3\alpha_5{+}2\alpha_6{+}\alpha_7\;\}.\]
Furthermore, there are four equivalence classes of pairs $(\Phi',w)\in\Xi$ 
under $\sim$, where $\Phi'=\Phi_0$; representatives $(\Phi_0,d_i)$, where 
$d_i\in\bW$ has minimal length in $\bW(\Phi_0)d_i$ for $i=1,2,3,4$, are 
given as follows.
\[\begin{array}{l@{\hspace{15pt}}c@{\hspace{15pt}}c} \hline 
d_i & \mbox{permutation} & \mbox{$\sigma_i'$-classes} \\\hline d_1=1_\bW & 
() & 50 \\ d_2=423143542654317654234 &  (1,6)(3,5)(4,7) & 10 \\ 
  d_3=4234542346542347654234  &  (1,7)(4,6) & 50 \\ 
  d_4=423143542314354654231435 42654 & (1,4)(3,5)(6,7) & 10 \\ \hline
\end{array}\]
(We use similar notational conventions as in the previous section.)

\begin{abs} {\bf The subgroup $\bH'=\langle \bT_0, \bU_\alpha\;(\alpha\in
\Phi_0)\rangle$.} \label{he7} As in \S \ref{he6}, every $t\in\bT_0$
has a unique expression $t=h(\xi_1,\ldots,\xi_7):=\prod_{1\leq i\leq 7} 
\alpha_i^\vee(\xi_i)$ where $\xi_i \in k^\times$ for $1\leq i\leq 7$. By 
\cite[Example~1.5.6]{gema}, we have 
\[Z(\bG)=\{h(1,\xi,1,1,\xi,1,\xi) \mid \xi=\pm 1\in k\} \cong \Z/2\Z.\]
A similar computation shows that 
\[ Z(\bH')=\{h(1,\pm 1,1,1,\xi,\xi^2,\xi^{-1}) \mid \xi\in k^\times, 
\xi^4=1\}\cong \Z/2\Z\times\Z/4\Z.\]
Since $q$ is not a power of $2$, we have $|Z(\bG)|=2$ and $|Z(\bH')|=8$. 
Given $t=h(1,\pm 1,1,1,\xi,\xi^2,\xi^{-1})\in Z(\bH')$ (where $\xi^4=1$), 
we have $C_\bG(t)=\bH'$ if and only if $\xi^2\neq 1$. The elements 
$h(1,-1,1,1,1,1,1)$ and $h(1,1,1,1,-1,1,-1)$ have a centraliser of 
type $D_6\times A_1$. Furthermore, one easily checks that 
\[ \dot{d}_2^{-1}t\dot{d}_2=zt,\qquad \dot{d}_3^{-1}t\dot{d}_3=t^{-1},
\qquad \dot{d}_4^{-1}t\dot{d}_4=zt^{-1},\]
where $z:=h\bigl(1,\xi^2,1,1,\xi^2,1,\xi^2\bigr) \in Z(\bG)$. 
These three relations show that all elements $t\in Z(\bH')$ such that
$C_\bG(t)=\bH'$ are conjugate in $\bG$. Thus, if $\cC$ denotes the 
$\bG$-conjugacy class of these elements, then 
\[ F(\cC)=\cC,\qquad \cC=\cC^{-1} \qquad\mbox{and}\qquad 
\cC=Z(\bG)\cC.\]
In order to see that $\cC$ is $F$-stable, we argue as follows. If $q\equiv 
1\bmod 4$, then $F(z)=z$ for all $z\in Z(\bH')$. On the other hand, if 
$q\equiv 3\bmod 4$, then $F(z)=z^{-1}$ for all $z\in Z(\bH')$ and, hence,
$Z(\bH')^F$ is a Klein four group. If $t\in Z(\bH')\cap \cC$ as above, 
then $F(t)=t^{-1}=\dot{d}_3^{-1} t\dot{d}_3 \in \cC$, as required.

As in \S \ref{se6}, we fix an element $s_1\in \cC^F$ and set 
$\bH_1':=C_\bG(s_1)$. We pick a regular unipotent element $u_1\in\bH_1'^F$ 
and let $\Sigma$ be the $\bG$-conjugacy class of $g_1:=s_1u_1$. Again,
we see that $\Sigma$ is $F$-stable and $Z(\bG)\Sigma=\Sigma=\Sigma^{-1}$. 
Furthermore, $A_\bG(g_1)\cong Z(\bH_1')$ has order~$8$ and 
\[A_\bG(g_1)=\langle \bar{s}_1,\bar{z}\rangle\quad \mbox{where $z$ 
is the non-trivial element of $Z(\bG)$}.\]
Now $F$ acts trivially on $A_\bG(g_1)$, regardless of the congruence 
class of $q$ modulo~$4$. Hence, the set $\Sigma^F$ always splits into 
$8$ conjugacy classes in $\bG^F$ (each with centraliser order 
$8q^7$), which are parametrised by the $8$ elements of $A_\bG(g_1)$. 
However, now it is less obvious whether we can choose $g_1\in\Sigma^F$ 
such that $g_1$ is conjugate to $g_1^{-1}$ in~$\bG^F$. We will come
back to this issue in \S \ref{g1e7}. (Since $Z(\bG)$ has 
order~$2$, we can not use the argument in Lemma~\ref{lemreg3}(b).)
\end{abs}

\begin{abs} {\bf Cuspidal unipotent character sheaves.} \label{ce7}
By an argument entirely analogous to that in \S \ref{ce6}
(but now using \cite[Prop.~20.5]{L2d} and its proof), we see
that there are two $F$-invariant cuspidal unipotent character sheaves
$A_1$ and $A_2$ on $\bG$. (Again, they are pulled back from 
$\tilde{\bG}=\bG/Z(\bG)$ via the canonical map $\pi\colon\bG\rightarrow
\tilde{\bG}$.) The local systems associated with $A_1$ and $A_2$
are one-dimensional; they correspond to linear characters $\psi_1$
and $\psi_2$ of $A_\bG(g_1)$ such that 
\[ \psi_1(\bar{s}_1)=\ii,\quad \psi_2(\bar{s}_1)=-\ii, \quad 
\psi_1(\bar{z})=\psi_2(\bar{z})=1\quad\mbox{for $z\in Z(\bH_1')^F$},\] 
where $\ii\in\K$ is fixed such that $\ii^2=-1$. (Recall that 
$\bar{g}_1=\bar{s}_1$.) Thus, $\psi_1$ and $\psi_2$ are completely 
determined, where $\psi_2$ is the complex conjugate of~$\psi_1$.
Furthermore, the roots of unity attached to $A_1$ and $A_2$ as in 
\S \ref{charf} are $\lambda_{A_1}=\ii$ and $\lambda_{A_2}=-\ii$. 
Using $\psi_1$ and $\psi_2$, we can now write down characteristic 
functions of $A_1$ and $A_2$, as in \S \ref{charf}. The values 
are given as follows, where $1\neq z\in Z(\bG)$.
\begin{center}
$\renewcommand{\arraystretch}{1.2} \renewcommand{\arraycolsep}{12pt} 
\begin{array}{|c|c|c|c|c|} 
\hline & \{\bar{1},\;\bar{z}\} & \{\bar{s}_1^2,\;\bar{s}_1^2\bar{z}\} 
& \{\bar{s}_1,\;\bar{s}_1\bar{z}\} &\{\bar{s}_1^{-1},\; \bar{s}_1^{-1}
\bar{z}\}\\
\hline \chi_{g_1,\psi_1} & q^{7/2} & -q^{7/2} & \ii q^{7/2} & 
-\ii q^{7/2} \\ \hline \chi_{g_1,\psi_2} & q^{7/2} & 
-q^{7/2} & -\ii q^{7/2} & \ii q^{7/2} \\ \hline  
\end{array}$
\end{center}
Note that, here, we have $\dim \bG- \dim \Sigma=\dim \bT_0=7$. 
We also assume that a square root $\sqrt{q}\in\K$ has been fixed (and
then $q^{n/2}:=\sqrt{q}^{\,n}$  for any $n\in \Z$).
\end{abs}

\begin{abs} {\bf Unipotent characters and almost characters.} \label{ue7}
Exactly as in \S \ref{ue6}, the unipotent characters of 
$\bG^F$ can be canonically identified with those of $\tilde{\bG}^F$,
where $\tilde{\bG}=\bG/Z(\bG)$. Again, they are parametrised by a 
certain set $X(\bW)$ (which only depends on $\bW$). We use the notation 
in the table on \cite[pp.~364--365]{LuB}; also note the special remarks
concerning type $E_7$ (and $E_8$) on \cite[p.~362]{LuB}. The unipotent 
almost characters are also parametrised by $X(\bW)$. The interesting 
cases for us are as follows, where $\xi=\ii\sqrt{q}\in\K$: 
\begin{align*}
R_{512_a'}=R_{(1,1)} &:= \textstyle{\frac{1}{2}}\bigl([512_a']+
[512_a] -E_7[\xi]-E_7[-\xi]\bigr),\\ R_{512_a}=R_{(1,\varepsilon)} &:= 
\textstyle{\frac{1}{2}}\bigl([512_a']+[512_a] +E_7[\xi]
+E_7[-\xi]\bigr),\\ R_{(g_2,1)} &:= 
\textstyle{\frac{1}{2}} \bigl([512_a']-[512_a] -E_7[\xi]
+E_7[-\xi]\bigr),\\ R_{(g_2,\varepsilon)} 
&:= \textstyle{\frac{1}{2}}\bigl([512_a']-[512_a] +E_7[\xi]
-E_7[-\xi]\bigr).
\end{align*}
Here, $512_a'$ and $512_a$ are irreducible characters of $\bW$; then 
$[512_a']$ and $[512_a]$ are the corresponding constituents
of $\mbox{Ind}_{\bB^F}^{\bG^F}(1)$; the characters $E_7[\pm \xi]$ 
are cuspidal unipotent. (Note also that the ``$g_2$'' in $(g_2,1)$ and
$(g_2,\varepsilon)$ has nothing to do with elements in $\bG^F$.) Now 
consider the two character sheaves $A_1$ and $A_2$ described above, with 
characteristic functions $\chi_{g_1,\psi_1}$ and $\chi_{g_1,\psi_2}$. By 
the main result of \cite[\S 4]{S3} (see also \cite[5.2]{S3}), there are 
scalars $\zeta,\zeta'\in\K$ of absolute value~$1$ such that
\[R_{(g_2,1)}=\zeta\chi_{g_1,\psi_1} \qquad \mbox{and}\qquad 
R_{(g_2,\varepsilon)}=\zeta' \chi_{g_1,\psi_2}.\]
(Again, in \cite[\S 4]{S3}, this is proved for $\tilde{\bG}$ 
but the discussion in \S \ref{ce7} shows that this also holds
for $\bG$, with $\psi_1$ and $\psi_2$ as above.)
By \cite[Table~1]{my03}, the characters $E_7[\xi]$ and $E_7[-\xi]$ 
are complex conjugate to each other, and their values lie in the field 
$\Q(\xi)$. Furthermore, all characters $[\phi]$ (where $\phi\in\Irr(\bW)$) 
have their values in $\Q(\sqrt{q})$; see \cite[Prop.~5.6]{my03}. We 
conclude that $R_{(g_2,1)}$ and $R_{(g_2, \varepsilon)}$ are complex 
conjugate to each other, and their values 
lie in $\Q(\ii, \sqrt{q})$. Since 
\[ R_{(g_2,1)}(g_1)=\zeta q^{7/2}\qquad\mbox{and}\qquad R_{(g_2,
\varepsilon)}(g_1)=\zeta' q^{7/2},\]
we can already conclude that $\zeta'=\overline{\zeta}$. 
\end{abs}

\begin{abs} {\bf On the choice of $g_1\in \Sigma^F$.} \label{g1e7}
We now come back to the issue of finding a ``good'' representative
$g_1\in \Sigma^F$. Recall that $Z(\bG)$ has order~$2$. By 
Example~\ref{correg2}, there are precisely two $\bG^F$-conjugacy 
classes $C,C'\subseteq \Sigma^F$ which have a non-empty intersection 
with $\bB^F\dot{w}_c\bB^F$. We let $g_1\in C\cup C'$. To fix the 
notation, we let $C$ be the $\bG^F$-conjugacy class associated with 
$\Sigma$ (and the choice of $\dot{w}_c$) as in Corollary~\ref{correg1};
by the construction in Example~\ref{correg2}, $C'$ is the 
$\bG^F$-conjugacy class parametrised by $\bar{z}\in A_\bG(g_1)$. The 
table in \S \ref{ce7} now shows that $\chi_{g_1,\psi_1}$ and
$\chi_{g_1,\psi_2}$ have the same value on all elements in $C\cup C'$. 
Using the formula in Example~\ref{reguni}, the (known) character table 
of $\bW$ and the required computations concerning $\sigma'$-conjugacy 
classes in $\bW$, we find that the restrictions of the almost 
characters $R_{512_a'}$ and $R_{512_a}$ to $\Sigma^F$ are 
identically zero. Finally, the relations in \S \ref{ue7} 
can be inverted and yield the following relations: 
\begin{align*}
[512_a']&=\textstyle{\frac{1}{2}}\bigl(R_{512_a'}+R_{512_a}+
\zeta\chi_{g_1,\psi_1}+\zeta'\chi_{g_1,\psi_2}\bigr),\\
[512_a]&=\textstyle{\frac{1}{2}}\bigl(R_{512_a'}+R_{512_a}-
\zeta\chi_{g_1,\psi_1}-\zeta'\chi_{g_1,\psi_2}\bigr),\\
E_7[\xi]&=\textstyle{\frac{1}{2}}\bigl(-R_{512_a'}+R_{512_a}-
\zeta\chi_{g_1,\psi_1}+\zeta'\chi_{g_1,\psi_2}\bigr),\\
E_7[-\xi]&=\textstyle{\frac{1}{2}}\bigl(-R_{512_a'}+R_{512_a}+
\zeta\chi_{g_1,\psi_1}-\zeta'\chi_{g_1,\psi_2}\bigr).
\end{align*}
So the above discussion implies that $E_7[\xi](g)=(\overline{\zeta}-
\zeta)q^{7/2}$ for all $g\in C\cup C'$. Next recall that $\Sigma=
\Sigma^{-1}$. So, by Lemma~\ref{lemreg3}(a), we have $\{C^{-1},C'^{-1}
\}= \{C,C'\}$ which implies that $E_7[\xi](g^{-1})=(\overline{\zeta}-
\zeta)q^{7/2}$ for all $g\in C\cup C'$. But the left hand side also 
equals $\overline{E_7[\xi](g)}=(\zeta-\overline{\zeta})q^{7/2}$. Hence, we
conclude that $\zeta'=\overline{\zeta}=\zeta=\pm 1$ and $E_7[\xi](g)=0$ for 
all $g\in C\cup C'$. Then all the values of $[512_a]$, $[512_a']$ and 
$E_7[\pm \xi]$ on $\Sigma^F$ are determined (up to $\zeta=\pm 1$); 
see Table~\ref{vale7a}.
\end{abs}

\begin{table}[htbp] \caption{Values of $[512_a']$, $[512_a]$ and
$E_7[\pm \xi]$ on $\Sigma^F$} \label{vale7a}
\begin{center}
$\renewcommand{\arraystretch}{1.2} \renewcommand{\arraycolsep}{10pt} 
\begin{array}{|c|c|c|c|c|} \hline (\zeta=\zeta'=\pm 1)
& \{\bar{1},\;\bar{z}\} & \{\bar{s}_1^2,\;\bar{s}_1^2\bar{z}\} & 
\{\bar{s}_1,\;\bar{s}_1\bar{z}\} &\{\bar{s}_1^{-1},\; \bar{s}_1^{-1}
\bar{z}\}\\ \hline
[512_a'] & \zeta q^{7/2} & -\zeta q^{7/2} & 0 & 0 \\ \hline 
[512_a] & -\zeta q^{7/2} & \zeta q^{7/2} & 0 & 0  \\ \hline 
E_7[\xi] & 0 & 0 & -\ii \zeta q^{7/2} & \ii \zeta q^{7/2} \\\hline 
E_7[-\xi] & 0 & 0 & \ii \zeta q^{7/2} &-\ii \zeta q^{7/2} \\\hline 
\multicolumn{5}{l}{\text{(The classes labelled by $\bar{1}$ and $\bar{z}$
correspond to $C$ and $C'$.)}}
\end{array}$
\end{center}
\end{table}

Now, how important is it to determine the scalar $\zeta=\pm 1$ exactly? 
The above expressions for $E_7[\pm \xi]$ show that $E_7[\xi](g)=E_7[-\xi]
(g)\in\Z$ for all $g\in\bG^F\setminus \Sigma^F$. In other words, the two 
characters $E_7[\xi]$ and $E_7[-\xi]$ can only (!) be distinguished by 
their values on elements in $\Sigma^F$, where they are given by 
Table~\ref{vale7a}. The same is also true for $[512_a']$ and $[512_a]$.
Thus, up to simultaneously exchanging the names of $[512_a']$,$[512_a]$ 
and of $E_7[\xi]$, $E_7[-\xi]$, we could assume ``without loss of 
generality'' that $\zeta=1$. For most applications of character theory, 
this is entirely sufficient.~---~However, it is actually possible to 
determine~$\zeta$ exactly.

\begin{prop} \label{hecke7} Recall that $g_1\in C\cup C'$ where $C$ and 
$C'$ are the two $\bG^F$-conjugacy classes that are contained in 
$\Sigma^F$ and have a non-empty intersection with $\bB^F\dot{w}_c
\bB^F$. Then $\zeta=\zeta'=1$, that is, $R_{(g_2,1)}=\chi_{g_1,\psi_1}$ 
and $R_{(g_2,\varepsilon)}=\chi_{g_1, \psi_2}$.
\end{prop}

\begin{proof} Let us denote the $\bG^F$-conjugacy classes in 
Table~\ref{vale7a} by $C_1,\ldots,C_8$ (from left to right); in 
particular, $C=C_1$ and $C'=C_2$. We will now try to evaluate the 
formula in Remark~\ref{hecke}(b) for elements $g\in \Sigma^F$. For 
this purpose, we write $X(\bW)=X^\circ\cup \{x_1,x_2\}$ where 
$x_1=(g_2,1)$ and $x_2=(g_2,\varepsilon)$. Since every unipotent 
character of $\bG^F$ is a linear combination of unipotent almost 
characters, we have 
\[[\phi]=[\phi]^\circ+\alpha_1(\phi)R_{x_1}+\alpha_2(\phi)R_{x_2}
\qquad\mbox{for each $\phi\in\Irr(\bW)$},\]
where $\alpha_1(\phi),\alpha_2(\phi)\in\K$ and $[\phi]^\circ$ is a linear 
combination of $\{R_x\mid x\in X^\circ\}$. Setting 
\[ B:=\sum_{\phi\in \Irr(\bW)} \phi^{(q)}(T_{w_c})
\alpha_1(\phi) \qquad\mbox{and}\qquad D(g):=\sum_{\phi\in\Irr(\bW)} 
\phi^{(q)}(T_{w_c})[\phi]^\circ(g) \]
for $g\in\Sigma^F$, we obtain 
\[\sum_{\phi\in\Irr(\bW)} \phi^{(q)}(T_w)[\phi](g)
=D(g)+ B{\cdot}\bigl(R_{x_1}(g)+R_{x_2}(g)\bigr).\]
Now we note the following. Let $x\in X^\circ$ and consider the possible
values of $R_x$ on $C_1,\ldots,C_8$. Since $R_x$ is a linear combination
of unipotent characters, $R_x$ takes the same value on $C_{2i-1}$, 
$C_{2i}$ for $i=1,\ldots,4$ (see Lemma~\ref{samev}). Now $R_x$ is
orthogonal to both $R_{x_1}$ and $R_{x_2}$. Hence, since the values of
$R_{x_1}$ on $C_1,\ldots,C_8$ are $1,1,-1,-1,\ii,\ii,-\ii,-\ii$ 
and those of $R_{x_2}$ are $1,1,-1,-1,-\ii,-\ii,\ii,\ii$, we 
conclude that the values of $R_x$ must be $x,x,x,x,y,y,y,y$, for 
some $x,y\in \K$. What is important here is that $D(g)$ takes the same 
value on all $g\in C_1\cup C_2\cup C_3\cup C_4$; let us denote by $D_0$ 
this common value of $D(g)$ on $C_1\cup C_2\cup C_3\cup C_4$. 

Next, the explicitly known matrix relating unipotent characters and 
unipotent almost characters shows that 
\[\textstyle\alpha_1(512_a')=\alpha_2(512_a')=\frac{1}{2},\qquad
\alpha_1(512_a)=\alpha_2(512_a)=-\frac{1}{2},\]
and $\alpha_1(\phi)=\alpha_2(\phi)=0$ for all $\phi\neq 512_a',512_a$.
Furthermore, we have 
\[(512_a')^{(q)}(T_{w_c})=q^{7/2}\qquad\mbox{and}\qquad (512_a)^{(q)}
(T_{w_c})=-q^{7/2},\]
by known results on character values of Hecke algebras (see 
\cite[Example~9.2.9(b)]{gepf}; these values are readily available 
within {\sf CHEVIE} \cite{chevie}). This yields $B=q^{7/2}$. 
Furthermore, $R_{x_1}$, $R_{x_2}$ take value $\zeta q^{7/2}$ on elements 
in $C_1\cup C_2$, and value $-\zeta q^{7/2}$ on elements in $C_3\cup C_4$.
Hence, we obtain:
\[\sum_{\phi\in\Irr(\bW)} \phi^{(q)}(T_w)[\phi](g)=\left\{
\begin{array}{cl} D_0+2\zeta q^7 & \qquad\mbox{if $g\in C_1\cup C_2$},\\ 
D_0-2\zeta q^7 & \qquad\mbox{if $g\in C_3\cup C_4$}. \end{array}\right.\]
By Remark~\ref{hecke}(b), the left hand side equals $4q^7$ or $0$, 
according to whether $g\in C_1\cup C_2$ or $g\in \Sigma^F\setminus (C_1
\cup C_2)$. Thus, $0=D_0-2\zeta q^7$ and, consequently, $4q^7=D_0+2\zeta 
q^7=4\zeta q^7$. In particular, $\zeta=1$.
\end{proof}

%%%%%%%%%%%%%%%%%%%%%%%%%%%%%%%%%%%%%%%%%%%%%%%%%%%%%%%%%%%%%%%%%%%%%%%%%%%
\section{Cuspidal character sheaves in type $F_4$} \label{secf4}

Throughout this section, let $\bG$ be simple of type $F_4$; here we have 
$Z(\bG)=\{1\}$. Let $q=p^f$ (where $f\geq 1$) be such that $F\colon 
\bG \rightarrow \bG$ defines an $\F_q$-rational structure. Here, $\bG$ is 
of split type; thus, $\sigma=\mbox{id}_\bW$ and the permutation $\alpha 
\mapsto\alpha^\dagger$ of $\Phi$ is the identity. Let $\Delta=\{\alpha_1,
\alpha_2,\alpha_3,\alpha_4\}$ be the set of simple roots in $\Phi^+$, where
the labelling is chosen as follows.
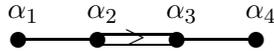
\begin{center}
\begin{picture}(210,15)
\put( 61,12){$\alpha_1$}
\put( 91,12){$\alpha_2$}
\put(121,12){$\alpha_3$}
\put(151,12){$\alpha_4$}
\put( 65,04){\circle*{6}}
\put( 95,04){\circle*{6}}
\put(125,04){\circle*{6}}
\put(155,04){\circle*{6}}
\put( 65,04){\line(1,0){30}}
\put( 95,02){\line(1,0){30}}
\put( 95,06){\line(1,0){30}}
\put(125,04){\line(1,0){30}}
\put(105,02){$>$}
\end{picture}
\end{center}
Except for \S \ref{f4p23} (at the very end), we will assume 
that $p\neq 2,3$.
The following subsystems of $\Phi$ occur in the classification of
cuspidal character sheaves; see \cite[Prop.~21.3]{L2d} and its proof.
\[\begin{array}{c@{\hspace{5pt}}c@{\hspace{10pt}}
l@{\hspace{5pt}}c@{\hspace{1pt}}c} \hline \Phi' & \Delta' & d_i & 
\mbox{perm.} & \mbox{$\sigma_i'$-classes} \\ \hline 
A_2{\times} A_2 & \alpha_1,\alpha_3,\alpha_4, \alpha_1{+}3\alpha_2{+}4
\alpha_3{+}2\alpha_4&  d_1=1_\bW & () & 9 \\
& & d_2=232432 &  (1,4)(2,3) & 9 \\\hline
A_3{\times} A_1 & \alpha_1,\alpha_2,\alpha_4, \alpha_1{+}2\alpha_2{+}4
\alpha_3{+}2\alpha_4&  d_1=1_\bW & () & 10 \\
& & d_2=3234323 & (1,4) & 10 \\\hline
B_4 & \alpha_1,\alpha_2,\alpha_3, \alpha_2{+}2\alpha_3{+}2\alpha_4&  
d_1=1_\bW & () & 20 \\\hline
C_3{\times} A_1 & \alpha_2,\alpha_3,\alpha_4, 2\alpha_1{+}3\alpha_2{+}4
\alpha_3{+}2\alpha_4&  d_1=1_\bW & () & 20 \\\hline
\end{array}\]
(We use the same notational conventions as in the previous sections.)
As before, we begin by working out the center of $\bH'
=\langle \bT_0, \bU_\alpha\; (\alpha\in \Phi')\rangle$ in each case.

\begin{abs} {\bf The subsystem $\Phi'$ of type $A_2\times A_2$.} 
\label{a2a2} In this case, we have 
\[Z(\bH')=\{h(1,1,\xi,\xi^{-1})\mid \xi \in k^\times,\xi^3=1\}.\]
Given $t\in Z(\bH')$, we have $C_\bG(t)=\bH'$ if and only if $t\neq 1$. 
Furthermore, one checks that $\dot{d}_2^{-1}t\dot{d}_2=t^{-1}$. Hence,
the two elements $t\in Z(\bH')$ such that $C_\bG(t)=\bH'$ are conjugate 
in $\bG$. If $\cC$ denotes the $\bG$-conjugacy class of these elements, 
then $\cC=\cC^{-1}$; furthermore, $F(\cC)=\cC$. Indeed, if $q\equiv 
1\bmod 3$, then $F(t)=t$ for all $t\in Z(\bH')$. On the other hand, if 
$q\equiv 2\bmod 3$, then $F(t)=t^{-1}$ for all $t\in Z(\bH')$. If 
$t\in Z(\bH')\cap \cC$, then $F(t)=t^{-1}=\dot{d}_2^{-1}t\dot{d}_2 \in 
\cC$, as required.
\end{abs}

\begin{abs} {\bf The subsystem $\Phi'$ of type $A_3\times A_1$.} 
\label{a3a1} In this case, we have 
\[Z(\bH')=\{h(1,1,\xi^2,\xi)\mid \xi \in k^\times,\xi^4=1\}.\]
Given $t=h(1,1,\xi^2,\xi)\in Z(\bH')$ (where $\xi^4=1$), we have 
$C_\bG(t)=\bH'$ if and only if $\xi^2\neq 1$. (Note that the element
$h(1,1,1,-1)$ has a centraliser of type~$B_4$.) Furthermore, one checks
that $\dot{d}_2^{-1}t\dot{d}_2=h(1,1,\xi^2,\xi^{-1})=t^{-1}$. Hence, 
the two elements $t\in Z(\bH')$ such that $C_\bG(t)=\bH'$ are conjugate 
in $\bG$. If $\cC$ denotes the $\bG$-conjugacy class of these elements, 
then $\cC=\cC^{-1}$; furthermore, $F(\cC)=\cC$. Indeed, if $q\equiv 
1\bmod 4$, then $F(t)=t$ for all $t\in Z(\bH')$. On the other hand, if 
$q\equiv 3\bmod 4$, then $F(t)=t^{-1}$ for all $t\in Z(\bH')$. If 
$t\in Z(\bH')\cap \cC$, then $F(t)=t^{-1}=\dot{d}_2^{-1}t\dot{d}_2 \in 
\cC$, as required.
\end{abs}

\begin{abs} {\bf The subsystem $\Phi'$ of type $B_4$.} \label{b4}
In this case, we have 
\[Z(\bH')=\{h(1,1,\xi,1)\mid \xi \in k^\times,\xi^2=1\}\cong \Z/2\Z.\]
Hence, if $s_1:=h(1,1,-1,1)\in Z(\bH')$, then $C_\bG(s_1)=\bH'$. (This
element $s_1$ is conjugate to the element $h(1,1,1,-1)$ mentioned above.)
\end{abs}

\begin{abs} {\bf The subsystem $\Phi'$ of type $C_3\times A_1$.} 
\label{c3a1} In this case, we have 
\[Z(\bH')=\{h(1,\xi,1,\xi)\mid \xi \in k^\times,\xi^2=1\}\cong \Z/2\Z.\]
Hence, if $s_1:=h(1,-1,1,-1)\in Z(\bH')$, then $C_\bG(s_1)=\bH'$. 
\end{abs}

\begin{abs} {\bf The unipotent class $F_4(a_3)$.} \label{f4a3} There are
seven cuspidal character sheaves on $\bG$; they are all unipotent
and $F$-invariant; see \cite[1.7]{L7} and \cite[\S 6, \S 7]{S2}. Let
$A$ be any of these seven cuspidal character sheaves and write
$A=\mbox{IC}(\Sigma,\cE)[\dim \cE]$ where $\Sigma$ is an $F$-stable
conjugacy class of $\bG$ and $\cE$ is an $F$-invariant irreducible local 
system on $\Sigma$. In all cases, $\cE$ is one-dimensional, so condition 
($*$) in \S \ref{charf} is satisfied. Furthermore, let $\cO_0$ be 
the conjugacy class of the unipotent part of an element in $\Sigma$. 
Then $\cO_0$ is the unipotent class denoted by $F_4(a_3)$ in 
\cite[\S 5]{Spa1}. (The identification of $\cO_0$ follows from 
\cite[Prop.~1.16]{L7} if $p$ is sufficiently large; by Taylor \cite{Tay15},
it is enough to assume that $p>3$. For small values of $p$, one 
can also use explicit computations in a matrix realisation of~$\bG$ 
and the results of Lawther \cite{law}.) We have $A_\bG(u)\cong \fS_4$ 
for $u\in\cO_0$, and there exists some $u\in \cO_0^F$ such that $F$ acts
trivially on $A_\bG(u)$; see Shoji \cite{Shof4}. Thus, $\cO_0^F$ splits
into five conjugacy classes in $\bG^F$, corresponding to the five conjugacy
classes of~$\fS_4$. As in \cite{Shof4}, we denote representatives of 
those five $\bG^F$-conjugacy classes by $x_{14},\ldots, x_{18}$. We 
have $|C_\bG^\circ(x_i)^F|=q^{12}$ in each case; furthermore, 
\begin{alignat*}{2}
A_\bG(x_{14})^F &\cong \fS_4, && \qquad \mbox{(cycle type $(1111)$)},\\
A_\bG(x_{15})^F &\cong D_8,&& \qquad \mbox{(cycle type $(22)$)},\\
A_\bG(x_{16})^F &\cong \Z/2\Z\times \Z/2\Z,&& \qquad \mbox{(cycle 
type $(211)$)},\\\ A_\bG(x_{17})^F &\cong \Z/4\Z,&& \qquad \mbox{(cycle 
type $(4)$)},\\ A_\bG(x_{18})^F &\cong \Z/3\Z,&& \qquad \mbox{(cycle 
type $(31)$)}.
\end{alignat*}
Thus, the five representatives $x_i$ ($i=14,\ldots,18$) can be 
distinguished from each other by the structure of the group $A_\bG(x_i)^F$.
Now let $n\in\Z$ be such that $p\nmid n$. Since $\cO_0$ is uniquely 
determined by its dimension, each $u\in\cO_0$ is conjugate to $u^n$ 
in~$\bG$; it then also follows that each $x_i$ is conjugate to $x_i^n$ 
in $\bG^F$, for $i=14,\ldots, 18$. We shall make use of this remark for
$n=2$ in the discussion below.
\end{abs}

Now we turn to the detailed description of the seven cuspidal character 
sheaves of $\bG$, where we follow Lusztig \cite[\S 20, \S 21]{L2d} and 
Shoji \cite[\S 6]{S2}. In each case, we will determine the scalar
$\zeta$ in the identity ($\clubsuit^\prime$); see \S \ref{ualm}.
We deal with the various cases in order of increasing difficulty.

\begin{abs} {\bf The cuspidal character sheaves $A_3$, $A_4$.} 
\label{csha34} Let $s_1\in\bG^F$ be semisimple such that $\bH_1'=
C_\bG(s_1)$ has a root system $\Phi'$ of type $A_2\times A_2$; recall 
from \S \ref{a2a2} that $Z(\bH_1') \cong \Z/3\Z$ and this is 
generated by~$s_1$. Let $u_1\in\bH_1'^F$ be a regular unipotent 
element and $\Sigma$ be the conjugacy class of $g_1:=s_1u_1$. As in 
\S \ref{se6}, one sees that $\Sigma=\Sigma^{-1}$ and that 
$A_\bG(g_1)\cong \Z/3\Z$ is generated by the image $\bar{g}_1$ of $g_1$ 
in $A_\bG(g_1)$. By \cite[(6.2.4)(c)]{S2}, there are two cuspidal character
sheaves $A_i=\mbox{IC}(\Sigma,\cE_i)[\dim \Sigma]$ where $i=3,4$. 
Let $1\neq \theta\in \K^\times$ be a fixed third root of unity. Then 
$\cE_3$ corresponds to the linear character $\psi_3 \colon A_\bG(g_1)
\rightarrow \K^\times$ such that $\psi_3(\bar{g}_1)=\theta$ and $\cE_4$
corresponds to the linear character $\psi_4\colon A_\bG(g_1)\rightarrow 
\K^\times$ such that $\psi_4(\bar{g}_1)=\theta^2$. By \cite[(6.4.1)]{S2}, 
$A_3$ is parametrised by $(g_3,\theta)\in X(\bW)$ and $A_4$ is
parametrised by $(g_3,\theta^2)\in X(\bW)$. By the main result of 
\cite[\S 6]{S2}, there are scalars $\zeta,\zeta'\in\K$ of absolute 
value~$1$ such that $R_{(g_3,\theta)}=\zeta\chi_{g_1,\psi_3}$ and 
$R_{(g_3,\theta^2)}=\zeta'\chi_{g_1,\psi_4}$, where the almost characters
$R_{(g_3,\theta)}$ and $R_{(g_3,\theta^2)}$ are defined as the following
linear combinations of unipotent characters:
\begin{align*}
R_{(g_3,\theta)} & :=\textstyle\frac{1}{3}\bigl([\phi_{12,4}]+
F_4^{\rm I\!I}[1]-[\phi_{6,6}']-[\phi_{6,6}'']+2F_4[\theta]-
F_4[\theta^2]\bigr),\\
R_{(g_3,\theta^2)} & :=\textstyle\frac{1}{3}\bigl([\phi_{12,4}]+
F_4^{\rm I\!I}[1]-[\phi_{6,6}']-[\phi_{6,6}''] 
-F_4[\theta]+2F_4[\theta^2]\bigr).
\end{align*}
Here, we use the notation in \cite[p.~479]{Ca2}, with analogous conventions
as in the previous sections. Thus, $\phi_{12,4}$ etc.\ are irreducible
characters of $\bW$; then $[\phi_{12,4}]$ etc.\ are the corresponding
irreducible constituents of $\mbox{Ind}_{\bB^F}^{\bG^F}(1)$; characters
denoted like $F_4^{\rm I\!I}[1]$ are cuspidal unipotent. (In this section
we refer to \cite{Ca2} instead of \cite{LuB}, because the full $21\times 
21$ Fourier matrix related to type $F_4$ is printed on \cite[p.~456]{Ca2}, 
and that matrix will be needed for several arguments below.) By 
Lemma~\ref{lemreg3}(b), we can choose $g_1\in \Sigma^F$ to be conjugate to 
$g_1^{-1}$ in $\bG^F$. By an argument analogous to that in 
\S \ref{ue6}, one sees that $R_{(g_3,\theta)}$ and $R_{(g_3,
\theta^2)}$ are complex conjugate to each other. So we conclude that 
\[\zeta=\zeta'=\pm 1\qquad\mbox{and}\qquad R_{(g_3,\theta)}(g_1)=
R_{(g_3,\theta^2)}(g_1)=\zeta q^2.\]
Inverting the matrix relating unipotent characters and unipotent almost
characters, we obtain the following relation:
\[F_4[\theta]=\textstyle\frac{1}{3}\bigl(R_{(12,4)}+R_{(1,\lambda^3)}
-R_{(6,6)'}-R_{(6,6)''}+2R_{(g_3,\theta)}-R_{(g_3,\theta^2)}\bigr),\]
where we just write, for example, $R_{(12,4)}$ instead of $R_{\phi_{(12,
4)}}$. Using {\sf CHEVIE} and the formula in Example~\ref{reguni}, we 
find that 
\[ R_{(12,4)}(g_1)=R_{(6,6)''}(g_1)=0, \qquad R_{(6,6)'}(g_1)=1.\]
By \cite[(6.2.2)]{S2}, the pair $(1,\lambda^3)$ also parametrises
a cuspidal character sheaf, which will be supported on a conjugacy
class distinct from $\Sigma$. By the main result of \cite[\S 6]{S2}, 
a characteristic function of that character sheaf equals $R_{(1,
\lambda^3)}$, up to multiplication by a scalar. Hence, we have
$R_{(1,\lambda^3)}(g_1)=0$ and we obtain $F_4[\theta](g_1)=
\frac{1}{3}(-1+\zeta q^2)$. Since the left hand side is an algebraic 
integer, this forces that $\zeta=1$. Thus, we have shown that 
\begin{equation*}
R_{(g_3,\theta)}=\chi_{g_1,\psi_3}\qquad\mbox{and}\qquad 
R_{(g_3,\theta^2)}=\chi_{g_1,\psi_4};\tag{a}
\end{equation*}
recall that, here, we fixed $g_1\in\Sigma^F$ such that $g_1$ is conjugate
to $g_1^{-1}$ in $\bG^F$. The values of $F_4[\theta]$ and $F_4[\theta^2]$
on $\Sigma^F$ are given by the following table.
\begin{center}
$\renewcommand{\arraystretch}{1.2} \renewcommand{\arraycolsep}{4pt} 
\begin{array}{|c|c|c|c|} 
\hline & \bar{1} & \bar{s}_1 & \bar{s}_1^2 \\ \hline F_4[\theta] & 
\frac{1}{3}(q^2-1) & \frac{1}{3}(q^2-1)+q^2\theta & \frac{1}{3}(q^2-1)+
q^2\theta^2 \\ \hline 
F_4[\theta^2] & \frac{1}{3}(q^2-1) & \frac{1}{3}
(q^2-1)+q^2\theta^2 & \frac{1}{3}(q^2-1)+q^2\theta \\ \hline
\end{array}$
\end{center}
This table also shows that $g_1\in\Sigma^F$ is uniquely determined
(up to $\bG^F$-conjugation) by the property that $g_1$ is conjugate
to $g_1^{-1}$ in $\bG^F$.
\end{abs}

\begin{abs} {\bf The cuspidal character sheaves $A_5$, $A_6$.} 
\label{csha56} Let $s_1\in\bG^F$ be semisimple such that $\bH_1'=
C_\bG(s_1)$ has a root system $\Phi'$ of type $A_3\times A_1$; 
recall from \S \ref{a3a1} that $Z(\bH_1') \cong \Z/4\Z$ and
this is generated by~$s_1$. Let $u_1\in\bH_1'^F$ be a regular unipotent 
element and $\Sigma$ be the conjugacy class of $g_1:=s_1u_1$. As above, 
one sees that $\Sigma=\Sigma^{-1}$ and that $A_\bG(g_1)\cong \Z/4\Z$ is 
generated by $\bar{g}_1\in A_\bG(g_1)$. By \cite[(6.2.4)(d)]{S2}, there 
are two cuspidal character sheaves $A_i=\mbox{IC}(\Sigma,\cE_i)[\dim 
\Sigma]$ where $i=5,6$; here, $\cE_5$ corresponds to the linear 
character $\psi_5 \colon A_\bG(g_1) \rightarrow \K^\times$ such that
$\psi_5(\bar{g}_1)=\ii$ (where $\ii^2=-1$ in $\K$) and $\cE_6$ 
corresponds to the linear character $\psi_6\colon A_\bG(g_1) \rightarrow 
\K^\times$ such that $\psi_6(\bar{g}_1)=-\ii$. By \cite[(6.4.1)]{S2}, 
$A_5$ is parametrised by $(g_4,\ii)\in X(\bW)$ and $A_6$ is parametrised 
by $(g_4,-\ii)\in X(\bW)$. By the main result of \cite[\S 6]{S2}, there 
are scalars $\zeta,\zeta'\in\K$ of absolute value~$1$ such that 
$R_{(g_4,\ii)}=\zeta\chi_{g_1,\psi_5}$ and $R_{(g_4,-\ii)}=\zeta'
\chi_{g_1,\psi_6}$, where
\begin{align*}
R_{(g_4,\ii)} & :=\textstyle\frac{1}{4}\bigl([\phi_{12,4}]-[\phi_{9,6}']
+[\phi_{1,12}']-F_4^{\rm I\!I}[1]-[\phi_{9,6}'']-F_4^{\rm I}[1]\\
&\qquad\qquad\qquad\qquad + [\phi_{1,12}'']+ [\phi_{4,8}]+2F_4[\ii]-
2F_4[-\ii]\bigr);
\end{align*}
there is a similar expression for $R_{(g_4,-\ii)}$ where the
roles of $F_4[\ii]$ and $F_4[-\ii]$ are interchanged.
By Lemma~\ref{lemreg3}(b), we can choose $g_1\in\Sigma^F$ to be conjugate 
in $\bG^F$ to $g_1^{-1}$. As in \S \ref{csha34}, we conclude 
that $\zeta=\zeta'=\pm 1$. Inverting the matrix relating 
unipotent characters and unipotent almost characters, we obtain:
\begin{align*}
F_4[\ii]&=\textstyle\frac{1}{4}\bigl(R_{(12,4)}-R_{(9,6)'}
+R_{(1,12)'}-R_{(1,\lambda^3)}-R_{(9,6)''}\\&\qquad\qquad\qquad
-R_{(g_2',\varepsilon)}+ R_{(1,12)''}+ R_{(4,8)}+2R_{(g_4,\ii)}-
2R_{(g_4,-\ii)}\bigr).
\end{align*}
Using Example~\ref{reguni}, we find that $R_\phi(g)=0$ for all $g\in
\Sigma^F$ and all $\phi\in\Irr(\bW)$ occurring in the sum on the right 
hand side. Again, by \cite[(6.2.2)]{S2}, the pair $(g_2',\varepsilon)$ 
also parametrises a cuspidal character sheaf, which will be supported 
on a conjugacy class distinct from $\Sigma$. By the main result of 
\cite[\S 6]{S2}, a characteristic function of that character sheaf 
equals $R_{(g_2',\varepsilon)}$, up to multiplication by a scalar. Hence, 
we also have $R_{(g_2',\varepsilon)}(g)=0$ for all $g\in \Sigma^F$. 
A similar argument shows that $R_{(1,\lambda^3)}(g_1)=0$. This yields 
the following table for the values of $F_4[\ii]$ on $\Sigma^F$:
\begin{center}
$\renewcommand{\arraystretch}{1.2} \renewcommand{\arraycolsep}{10pt} 
\begin{array}{|c|c|c|c|c|} 
\hline & \bar{1} & \bar{g}_1^2 & \bar{g}_1 &\bar{g}_1^{-1} \\ \hline
F_4[\ii] & 0 & 0 & \zeta \ii q^2 & -\zeta \ii q^2 \\ \hline 
F_4[-\ii] & 0 & 0 & -\zeta \ii q^2 & \zeta \ii q^2 \\ \hline 
\end{array}$
\end{center}
We can now draw the following conclusions. Let $C_1,C_2,C_3,C_4$ be 
the four conjugacy classes of $\bG^F$ into which $\Sigma^F$ splits (not 
necessarily ordered as in the above table). We have $\Sigma=\Sigma^{-1}$, 
so taking inverses permutes the four classes. The table shows that we 
can arrange the notation such that $C_4=C_3^{-1}$, where $C_3=
\mbox{Sh}_\bG(C_1)$ and $C_4=\mbox{Sh}_\bG(C_2)$; see 
\S \ref{paracl}(a).  Since $g_1\in\Sigma^F$ 
is conjugate in $\bG^F$ to $g_1^{-1}$, this forces that $C_1=C_1^{-1}$, 
$C_2=C_2^{-1}$ and $g_1\in C_1\cup C_2$. By Lemma~\ref{lemreg3}(b), we 
can further fix the notation such that $C_1 \cap \bB^F\dot{w}_c
\bB^F\neq \varnothing$ and $C_2\cap \bB^F\dot{w}_c\bB^F=
\varnothing$. Then we claim:
\begin{equation*}
\zeta=\left\{\begin{array}{rl} 1 & \qquad \mbox{if $g_1\in C_1$},\\
-1 & \qquad \mbox{if $g_1\in C_2$}.\end{array}\right.\tag{a}
\end{equation*}
This is seen by an argument entirely analogous to the proof of
Proposition~\ref{hecke7}, based on the formula in 
Remark~\ref{hecke}(b). The data required for that argument
(that is, the constants $\alpha_1(\phi)$, $\alpha_2(\phi)$ and
the values $\phi^{(q)}(T_{w_c})$) are now given as follows. 
We have $\alpha_1(\phi)=\frac{1}{4}$ for $\phi\in\{\phi_{1,12}', 
\phi_{1,12}'', \phi_{4,8}, \phi_{12,4}\}$, and $\alpha_1(\phi)=0$ 
otherwise; furthermore, if $\alpha_1(\phi)\neq 0$, then 
$\phi^{(q)}(T_{w_c})=q^2$; we omit further details. 
\end{abs}

\begin{abs} {\bf The cuspidal character sheaf $A_1$.} \label{csha1}
Let $\Sigma$ be the unipotent class of $\bG$ denoted by $F_4(a_3)$,
as already introduced in \S \ref{f4a3}. We take $g_1:=x_{14}
\in\Sigma^F$; hence, $F$ acts trivially on $A_\bG(g_1) \cong \fS_4$. We 
also remarked in \S \ref{f4a3} that $g_1$ is conjugate in $\bG^F$ 
to $g_1^{-1}$. As in \cite[(6.2.4)(a)]{S2}, there is a cuspidal 
character sheaf $A_1=\mbox{IC}(\Sigma,\cE)[\dim \Sigma]$ where $\cE$ 
corresponds to the sign character $\mbox{sgn}\in \Irr(A_\bG(g_1))$. 
In \cite[\S 6]{S2}, it is not stated explicitly to which parameter in 
$X(\bW)$ the character sheaf $A_1$ corresponds, but this is easily 
found as follows, using the information already available from 
\S \ref{csha34}. (See also the argument in 
\cite[Lemma~8.8]{L3}.) We claim that $A_1$ is parametrised by 
$(1,\lambda^3) \in X(\bW)$. Assume, if possible, that this is not the 
case. By Shoji's results \cite{Sh82} on the Green functions of
$\bG^F$, we can compute $R_\phi(g_1)$ for any $\phi\in \Irr(\bW)$.
(These results are known to hold whenever $p\neq 2,3$; see 
\cite[Theorem~5.5]{S3} and \cite[\S 3]{myav}.) In  particular, we 
obtain $R_{(6,6)''}(g_1)=0$, $R_{(6,6)'}(g_1)=2q^4$ and $R_{(12,4)}
(g_1)=q^4$. Since $R_{(g_3,\theta)}$ and $R_{(g_3,\theta^2)}$ are zero 
on unipotent elements (see \S \ref{csha34}), we conclude that 
$F_4[\theta](g_1)=\frac{1}{3}(q^4-2q^4)=-q^4/3$, contradiction since 
$p\neq 3$.  Hence, $A_1$ is parametrised by $(1,\lambda^3)$. By the main 
result of \cite[\S 6]{S2}, there is a scalar $\zeta\in\K$ of absolute 
value~$1$ such that $R_{(1,\lambda^3)}=\zeta\chi_{g_1,\text{sgn}}$.
The exact expression of $R_{(1,\lambda^3)}$ as a linear combination
of $21$ unipotent characters is obtained from the Fourier matrix on
\cite[p.~456]{Ca2} and the list of labels for unipotent characters on
\cite[p.~479]{Ca2}; we will not print it here. 

It was first shown by Kawanaka \cite[\S 4]{K2} that $\zeta=1$, assuming
that $p,q$ are sufficiently large; Lusztig \cite[8.6, 8.12]{L3} shows this 
assuming that $q$ satisfies a certain congruence condition. Since Kawanaka's
results on generalised Gelfand--Graev representations are now known to hold 
whenever $q$ is a power of a good prime~$p$ (see Taylor \cite{Tay15}), 
we can conclude that $\zeta=1$ holds unconditionally (but recall our
standing assumption that $p>3$). 

We can also argue as follows. Consider again the formula for 
$F_4[\ii]$ in \S \ref{csha56}. Using Shoji's results
on Green functions, we can compute the values of $R_\phi$ on unipotent 
elements, for all~$\phi$ occurring in that formula. Furthermore, we 
have $R_{(g_2',\varepsilon)}(g_1)=R_{(g_4,\pm \ii)}(g_1)=0$. This 
yields that $F_4[\ii](g_1)=-\frac{1}{4}q^4(\zeta q^2-1)$. Since $g_1$ 
is $\bG^F$-conjugate to $g_1^{-1}$, we have $\zeta=\overline{\zeta}$. 
Since $F_4[\ii](g_1)$ is an algebraic integer, we must have $\zeta=1$.
\end{abs}

\begin{table}[htbp] \caption{Some character values on the unipotent
class $F_4(a_3)$} \label{tabf4a3}
\begin{center}
$\renewcommand{\arraystretch}{1.3} \renewcommand{\arraycolsep}{1pt} 
\begin{array}{|c|c|c|c|c|c|} \hline & x_{14} \;(1111) & x_{15} \;(22) 
& \;x_{16} \;(211) & x_{17} \;(4)& x_{18}\; (31)\\ \hline 
[\phi_{1,12}'] & \frac{1}{8}q^4(q^2{-}1){+}3q^4 &\frac{1}{8}q^4(q^2{-}1) & 
\frac{1}{8} q^4(1{-}q^2) & \frac{1}{8}q^4(1{-}q^2) & \frac{1}{8}q^4(q^2{-}1) 
\\ \hline [\phi_{1,12}''] & \frac{1}{8}q^4(q^2{-}1) &\frac{1}{8}q^4
(q^2{-}1){+} q^4 & \frac{1}{8} q^4(1{-}q^2) & \frac{1}{8}q^4(1{-}q^2) & 
\frac{1}{8}q^4(q^2{-}1)\\ \hline F_4[{-}1] & \frac{1}{4}q^4(1{-}q^2) 
&\frac{1}{4}q^4(1{-}q^2) & \frac{1}{4}q^4(q^2{-}1){+}q^4 & \frac{1}{4}
q^4(q^2{-}1) & \frac{1}{4}q^4(1{-}q^2) \\ \hline F_4[\ii] & 
\frac{1}{4}q^4(1{-}q^2) &\frac{1}{4}q^4(1{-}q^2) & 
\frac{1}{4}q^4(q^2{-}1) & \frac{1}{4}q^4(q^2{-}1){+}q^4 & 
\frac{1}{4}q^4(1{-}q^2) \\ \hline \end{array}$
\end{center}
\end{table}

\begin{abs} {\bf Character values on $F_4(a_3)$.} \label{f4a3b}
Once $R_{(1,\lambda^3)}$ has been determined, we can determine all 
character values on $\bG_{\text{uni}}^F$, where $\bG_{\text{uni}}$ 
denotes the unipotent variety of $\bG$. Indeed, the $25$ unipotent almost 
characters $R_\phi$ (for $\phi\in \Irr(\bW)$) remain linearly 
independent upon restriction to $\bG_{\text{uni}}^F$; they are 
explicitly computed by Shoji \cite{Sh82}. (As mentioned above, Shoji's 
results remain valid whenever $p\neq 2,3$.) Hence, together with the 
``cuspidal'' almost character $R_{(1,\lambda^3)}$, we obtain $26$ 
linearly independent functions on $\bG_{\text{uni}}^F$. Since there 
are also $26$ unipotent conjugacy classes of $\bG^F$ (see 
\cite[Theorem~2.1]{Shof4}), we obtain a basis for the space of class 
functions on $\bG_{\text{uni}}^F$. Note that all the remaining unipotent
almost characters are orthogonal to the functions in that basis, which
implies that they are identically zero on $\bG_{\text{uni}}^F$. In 
\S \ref{csha2} below, we shall need the values of some 
unipotent characters on~$\cO_0^F$, with $\cO_0$ as in \S \ref{f4a3}.
The values displayed in Table~\ref{tabf4a3} will allow us to distinguish 
the various $\bG^F$-conjugacy classes contained in $\cO_0^F$. These 
values are easily obtained from the functions {\tt UnipotentCharacters} 
and {\tt ICCTable} in Michel's version of {\sf CHEVIE} \cite{jmich}. 
\end{abs}

\begin{abs} {\bf The cuspidal character sheaf $A_7$.} \label{csha7}
Let $s_1:=h(1,1,-1,1)\in\bT_0^F$ and $\bH_1':=C_\bG(s_1)$; then $\bH_1'$
has a root system $\Phi'$ of type $B_4$ (see \S \ref{b4}); recall 
that $Z(\bH_1') \cong \Z/2\Z$ and this is generated by~$s_1$. Consider 
the natural isogeny $\beta\colon \bH_1' \rightarrow \overline{\bH}_1':=
\mbox{SO}_9(k)$ (defined over $\F_q$). Let $\cO$ be the unipotent class
of $\bH_1'$ such that the elements $\beta(u)\in \overline{\bH}_1'$, for 
$u\in \cO$, have Jordan type $(5,3,1)$. Let $\Sigma$ be the conjugacy 
class of $s_1u$, where $u\in \cO$. Now $\cO$ is $F$-stable and so 
$\Sigma$ is also $F$-stable. By Shoji \cite[Table~4]{Shof4}, and the
correction discussed by Fleischmann--Janiszczak \cite[p.~233]{flja}, we 
have:
\begin{equation*}
\begin{array}{l} \mbox{There exists an element $g_1\in \Sigma^F$ such 
that $A_\bG(g_1)$ is dihedral of}\\\mbox{order $8$ and $F$ acts trivially 
on $A_\bG(g_1)$; we have $|C_\bG(g_1)^F|=8q^8$.}\end{array} \tag{a}
\end{equation*}
%This is seen as follows. For $u\in \cO$ we have $A_{\overline{\bH}_1'}
%(\beta(u))\cong \Z/2\Z\times\Z/2\Z$ and $A_{\bH_1'}(u)$ is a dihedral group 
%of order~$8$; see \cite[14.3, Case~(b)]{LuIC}. Furthermore, $C_{\bH_1'}^\circ
%(u)$ is a unipotent group of dimension~$8$. Now one can choose $u_0\in\cO$ 
%such that $F(u_0)=u_0$ and $\beta(u_0)$ is a ``split'' element, as defined 
%by Shoji \cite[2.9]{S6}. In particular, $F$ acts trivially on 
%$A_{\overline{\bH}_1'}(\beta(u_0))$ and, hence, $F$ will also act trivially 
%on $A_{\bH_1'}(u_0)$. Setting $g_1=s_1u_0$, we have $A_\bG(g_1)
%\cong A_{\bH_1'}(u_0)$ and (a) follows. 
%Once it is established that $A_\bG(g_1)$ is a dihedral group of order~$8$
%(with trivial $F$-action), we see that 
Thus, the set $\Sigma^F$ splits into five
conjugacy classes in $\bG^F$, with centraliser orders $8q^8,8q^8,4q^4,4q^4,
4q^4$. So there are two possibilities for the $\bG^F$-conjugacy class
of~ $g_1$ as in (a). (We just choose one of them; this choice does not
affect the result at the end. By \S \ref{paracl}(a) and 
\cite[Chap.~II, Prop.~7.2]{DiMi0}, we also see that the two classes are
interchanged by the Shintani map $\mbox{Sh}_\bG$.) Now, by 
\cite[(6.2.4)(e)]{S2}, there is a cuspidal character sheaf $A_7=
\mbox{IC}(\Sigma,\cE)[\dim \Sigma]$ where $\cE$ corresponds to the 
sign character $\mbox{sgn}\in \Irr(A_\bG(g_1))$. By \cite[(6.2.2)]{S2}, 
$A_7$ is parametrised either by the pair $(g_2,\varepsilon)$ or by the
pair $(g_2',\varepsilon)$ in $X(\bW)$. We can easily fix this as 
follows. We note that the eigenvalue $\lambda_{A_7}=\mbox{sgn}
(\bar{g}_1)$ in \S \ref{charf} must be~$1$ since $\bar{g}_1$ 
is in the center of $A_\bG(g_1)$. But there are also certain eigenvalues 
for the almost characters, where $\lambda_{R_x}=-1$ for $x=(g_2,
\varepsilon)$, and $\lambda_{R_x}=1$ for $x=(g_2',\varepsilon)$; see 
\cite[(6.2.2)]{S2}. By the main result of \cite[\S 6]{S2}, there is 
a scalar $\zeta\in\K$ of absolute value~$1$ such that $R_x=\zeta
\chi_{(g_1,\text{sgn})}$ where $x\in\{(g_2,\varepsilon), (g_2',
\varepsilon)\}$ and where the eigenvalues of the character sheaves do 
match those of the almost characters (see \cite[4.6]{S3}). So we 
conclude that 
\begin{equation*}
\mbox{$A_7$ is parametrised by $(g_2',\varepsilon)\in X(\bW)$ and 
$R_{(g_2',\varepsilon)}=\zeta\chi_{g_1,\text{sgn}}$}.\tag{b}
\end{equation*} 
The exact expression of $R_{(g_2',\varepsilon)}$ as a linear combination
of $18$ unipotent characters is obtained from the Fourier matrix on
\cite[p.~456]{Ca2} and the list of labels for unipotent characters on
\cite[p.~479]{Ca2}; we will not print it here. We claim:
\begin{equation*}
\mbox{With $g_1$ as in (a), we have $\zeta=1$}.\tag{c}
\end{equation*}
This is seen as follows. We use again the following identity from 
\S \ref{csha56}:
\begin{align*}
F_4[\ii]&=\textstyle\frac{1}{4}\bigl(R_{(12,4)}-R_{(9,6)'}
+R_{(1,12)'}-R_{(1,\lambda^3)}-R_{(9,6)''}\\&\qquad\qquad\qquad
-R_{(g_2',\varepsilon)}+ R_{(1,12)''}+ R_{(4,8)}+2R_{(g_4,\ii)}-
2R_{(g_4,-\ii)}\bigr).
\end{align*}
Now all $R_\phi$ ($\phi\in \Irr(\bW)$) are rational-valued. Since 
$\overline{F_4[\ii]}=F_4[-\ii]$ and $\overline{R_{(g_4,\ii)}}=
R_{(g_4,-\ii)}$, we conclude that $R_{(g_2',\varepsilon)}$ is invariant 
under complex conjugation and, hence, $\zeta=\pm 1$. Now evaluate 
$F_4[\ii]$ on $g_1\in\Sigma^F$. Note that $R_{(1,\lambda^3)}$ and
$R_{(g_4,\pm \ii)}$ have support on conjugacy classes that are 
distinct from~$\Sigma^F$ and, hence, their value is zero on~$g_1$; 
see \S \ref{csha56} and \S \ref{csha1}. By Example~\ref{secf4g}, 
we obtain  
\begin{gather*}
R_{(9,6)''}(g_1)=R_{(4,8)}(g_1)=R_{(1,12)''}(g_1)=q^2,\\
R_{(12,4)}(g_1)= R_{(9,6)'}(g_1)= R_{(1,12)'}(g_1)=0.
\end{gather*}
(Recall that $g_1$ is chosen such that $|C_\bG(g_1)^F|=8q^8$.) Since
$R_{(g_2',\varepsilon)}(g_1)=\zeta q^4$, we obtain $F_4[\ii](g_1)=
\frac{1}{4}q^2(1-\zeta q^2)$. Since the left hand side is an algebraic 
integer, we deduce that $\zeta=1$. Thus, (c) is proved. Finally, we note:
\begin{equation*}
\mbox{If $g_1=s_1u_1$ is as in (a), then $u_1$ is $\bG^F$-conjugate
to $x_{14}$ or $x_{15}$}.\tag{d}
\end{equation*} 
Indeed, since $C_\bG(g_1)\subseteq C_\bG(g_1^2)=C_\bG(u_1^2)$ and
$u_1^2$ is $\bG^F$-conjugate to $u_1\in \cO_0$, we conclude that $8$ 
divides $|C_\bG(u_1)^F|$. Hence, the only possibilities are that
$u_1$ is $\bG^F$-conjugate to $x_{14}$ or $x_{15}$. I conjecture that
for one of the two possibilities of $g_1=s_1u_1$ as in (a), we do have
that $u_1$ is $\bG^F$-conjugate to $x_{14}$ (but the choice of that
$g_1$ may depend on $q\bmod 4$).
\end{abs}

\begin{abs} {\bf The cuspidal character sheaf $A_2$.} \label{csha2}
Let $s_1\in\bG^F$ be semisimple such that $\bH_1'=C_\bG(s_1)$ has a 
root system $\Phi'$ of type $C_3\times A_1$; recall from 
\S \ref{c3a1} that $Z(\bH_1') \cong \Z/2\Z$ and this is 
generated by~$s_1$. Now we have a natural isogeny $\beta \colon
\mbox{Sp}_4(k) \times \mbox{SL}_2(k)\rightarrow \bH_1'$ (defined over 
$\F_q$). Let $\cO$ be the unipotent class of $\bH_1'$ that corresponds 
to unipotent elements of Jordan type $(4,2)\times (2)$ under~$\beta$.
We start by picking any element $u_1\in\cO^F$ and let $\Sigma$ be the 
conjugacy class of $g_1:=s_1u_1$. We have $\dim \bG-\dim \Sigma=6$ 
and $|C_\bG(g_1)^F|=4q^6$; one easily sees that $\Sigma=\Sigma^{-1}$. 
Now there is some $1\neq a\in A_\bG(g_1)$ such that 
\[A_\bG(g_1)=\langle \bar{g}_1\rangle\times \langle a\rangle\cong
\Z/2\Z\times \Z/2\Z\qquad \mbox{(with trivial $F$-action)}.\]
By \cite[(6.2.4)(b)]{S2}, there is a cuspidal character sheaf $A_2=
\mbox{IC}(\Sigma,\cE)[\dim \Sigma]$ where $\cE$ corresponds to a
non-trivial $\psi \in \Irr(A_\bG(g_1))$ (further specified below). By 
\cite[(6.2.2)]{S2}, $A_2$ is parametrised either by the pair $(g_2,
\varepsilon)$ or by the pair $(g_2',\varepsilon)$ in $X(\bW)$. By 
\S \ref{csha7}(b), we conclude that $A_2$ must be parametrised 
by $(g_2,\varepsilon)$; in particular, $\psi(\bar{g}_1)=\lambda_{A_2}=-1$.
We can now also fix the element $a\in A_\bG(g_1)$ such that $\psi(a)=1$. 
By the main result of \cite[\S 6]{S2}, there is a scalar $\zeta\in\K$ of 
absolute value~$1$ such that $R_{(g_2,\varepsilon)}=\zeta\chi_{g_1,\psi}$, 
where
\begin{align*}
R_{(g_2,\varepsilon)}&:=\textstyle\frac{1}{4}\bigl([\phi_{12,4}]+
[\phi_{9,6}']-[\phi_{1,12}']-F_4^{\rm I\!I}[1]-2[\phi_{16,5}]
\\&\qquad\quad +2F_4[-1] +[\phi_{9,6}'']-F_4^{\rm I}[1]-
[\phi_{1,12}'']+ [\phi_{4,8}]\bigr).
\end{align*}
Let $C_1,C_2,C_3,C_4$ be the four $\bG^F$-conjugacy classes into which
$\Sigma^F$ splits (initially ordered in no particular way). For each~$i$,
we denote $C_i^{[2]}:=\{g^2\mid g\in C_i\}$. Writing $g_1=s_1u_1$ as 
above, we have $u_1\in\cO_0$; furthermore, $u_1,u_1^2$ are 
$\bG^F$-conjugate and so $g_1^2=u_1^2\in \cO_0$ (see 
\S \ref{f4a3}). We claim that the notation can be arranged such that 
\begin{equation*}
x_{14}\in C_1^{[2]}, \qquad x_{15}\in C_2^{[2]} \qquad \mbox{and}
\qquad x_{16}\in C_3^{[2]}= C_4^{[2]}.\tag{a}
\end{equation*}
Depending on how we choose $g_1\in \Sigma^F$, the scalar $\zeta$
is then determined as follows.
\begin{equation*}
\zeta=\left\{\begin{array}{rl} 1 & \qquad \mbox{if $g_1\in C_1\cup
C_2$},\\ -1 & \qquad \mbox{if $g_1\in C_3\cup C_4$}.
\end{array}\right.\tag{b}
\end{equation*}
This is proved as follows. Inverting the matrix relating unipotent 
characters and unipotent almost characters, we obtain:
\begin{align*}
F_4[-1]&=\textstyle\frac{1}{4}\bigl(R_{(12,4)}+R_{(9,6)'}
-R_{(1,12)'} -R_{(1,\lambda^3)}-2R_{(16,5)} \\&\qquad\quad 
+2R_{(g_2,\varepsilon)} +R_{(9,6)''}-R_{(g_2',\varepsilon)}
-R_{(1,12)''}+ R_{(4,8)}\bigr).
\end{align*}
Now we evaluate this on $g_1\in \Sigma^F$. By \S \ref{csha1} and \S 
\ref{csha7}, we have $R_{(1,\lambda^3)}(g_1)=R_{(g_2',\varepsilon)}(g_1)=0$.
By a computation entirely analogous to that in Example~\ref{secf4g}, 
we obtain $R_{(16,5)}(g_1)=q$ and 
\[ R_{(12,4)}(g_1)= R_{(9,6)'}(g_1)= R_{(9,6)''}(g_1)= R_{(1,12)'}(g_1)=
R_{(1,12)''}(g_1)=0;\]
this does not depend on how we choose $g_1\in\Sigma^F$. Since 
$R_{(g_2,\varepsilon)}$ takes the values $\zeta q^3,\zeta q^3, -\zeta 
q^3,-\zeta q^3$ on the representatives in $\Sigma^F$ parametrised by
$\bar{1},a, \bar{g}_1,a\bar{g}_1$, this yields the following values for 
$F_4[-1]$ on $\Sigma^F$.
\begin{center}
$\renewcommand{\arraystretch}{1.2} \renewcommand{\arraycolsep}{5pt} 
\begin{array}{|c|c|c|c|c|} \hline & \bar{1}&a&\bar{g}_1&a\bar{g}_1\\
\hline F_4[-1] & \frac{1}{2}q(\zeta q^2-1) & \frac{1}{2}q(\zeta q^2-1) & 
-\frac{1}{2}q(\zeta q^2+1) &-\frac{1}{2}q(\zeta q^2+1)  \\ \hline 
\end{array}$
\end{center}
By \cite[Table~1]{my03}, the character $F_4[-1]$ is rational-valued,
so we must have $\zeta=\pm 1$. Regardless of whether $\zeta$ equals 
$1$ or $-1$, two of the above values are $\frac{1}{2}q(q^2-1)$,
and two of them are $-\frac{1}{2}q(q^2+1)$. Thus, two of the above 
values are even integers, and two of them are odd integers. 
Now compare with Table~\ref{tabf4a3}: 
\[ F_4[-1](x_{16})\equiv 1\bmod 2 \qquad\mbox{and}\qquad 
F_4[-1](x_{i}) \equiv 0 \bmod 2 \quad \mbox{for $i\neq 16$}.\]
By a well-known fact from the general character theory of finite 
groups, we have $F_4[-1](g_1^2)\equiv F_4[-1](g_1) \bmod 2$.
Hence, if $g_1\in\Sigma^F$ is such that $F_4[-1](g_1)$ is odd, 
then $g_1^2$ must be $\bG^F$-conjugate to $x_{16}$. Since there are 
two $\bG^F$-conjugacy classes in $\Sigma^F$ on which the value of 
$F_4[-1]$ is odd, we conclude that $x_{16}\in C_i^{[2]}$ for two 
values of $i\in\{1,2,3,4\}$; we arrange the notation such that these 
two values are $i=3$ and $i=4$. Now choose $g_1\in\Sigma^F$ such 
that $g_1\in C_3\cup C_4$. Since $F_4[-1](g_1)$ is given by the
entry corresponding to $\bar{1}\in A_\bG(g_1)$ in the above table,
we conclude that $\frac{1}{2}q(\zeta q^2-1)$ must be odd and so 
$\zeta=-1$. Thus, (a) and (b) are proved as far as 
$C_3$ and $C_4$ are concerned. On the other hand, let us choose $g_1
\in\Sigma^F\setminus (C_3\cup C_4)$. Then $F_4[-1](g_1)=\frac{1}{2}q
(\zeta q^2-1)$ must be even and so $\zeta=1$. So all that remains to 
be done is to identity $i,j\in\{14,\ldots,18\}$ such that 
$x_i\in C_1^{[2]}$ and $x_j\in C_2^{[2]}$. For this purpose, we 
consider the characters $[\phi_{1,12}']$ and $[\phi_{1,12}'']$. 

Using the ingredients of the {\sf CHEVIE} function {\tt LusztigMapb} 
explained in \cite[\S 7]{jmich} (which relies on the theoretical fact 
that the indicator function of a $\bG^F$-conjugacy class is ``uniform'', 
see \cite[\S 8]{mylaus}), we can compute $\sum_{g \in \Sigma^F} \rho(g)$
for any $\rho\in \Unip(\bG^F)$. Since all elements in $\Sigma^F$
have the same centraliser order, we can actually compute the sum
of the four values of $\rho$ on $C_1,C_2,C_3,C_4$. Applying this to 
$\rho=[\phi_{1,12}']$, we find that the result is~$-q$. Consequently, 
the four values of $[\phi_{1,12}']$ on $C_1,C_2,C_3,C_4$ cannot all 
have the same parity. Hence, there exists some $g\in \Sigma^F$ such 
that $[\phi_{1,12}'](g)\equiv [\phi_{1,12}'](x_{14}) \bmod 2$. But 
then we also have 
\[[\phi_{1,12}'](g^2) \equiv [\phi_{1,12}'](g)\equiv 
[\phi_{1,12}'](x_{14}) \bmod 2.\]
Since $[\phi_{1,12}'](x_i) \not\equiv [\phi_{1,12}'](x_{14}) \bmod 2$
for $i\neq 14$ (see Table~\ref{tabf4a3}), we conclude that $g^2$ 
is $\bG^F$-conjugate to $x_{14}$. Thus, we can arrange the notation such 
that $x_{14} \in C_1^{[2]}$. Then a completely analogous argument using 
the character $[\phi_{1,12}'']$ shows that $x_{15} \in C_2^{[2]}$. Thus, 
(a) and (b) are proved. The above table of values also shows that the 
values of $F_4[-1]$ on the classes parametrised by $\bar{1}$ and 
$\bar{g}_1$ have a different parity; similarly for~$a$ and~$a\bar{g}_1$. 
Hence, we can fix the notation for $C_3$ and $C_4$ such that $C_3=
\mbox{Sh}_\bG(C_1)$ and $C_4=\mbox{Sh}_\bG(C_2)$ (see 
\S \ref{paracl}(a)). 

Finally, we remark that we can also obtain an explicit representative
in $\Sigma^F$. Indeed, using {\sf CHEVIE}, we can easily compute the
full $\bW$-orbit of $s_1$; by inspection, $s_1':=h(-1,-1,1,1) \in
\bT_0^F$ belongs to that orbit, that is, $s_1'$ is conjugate to~$s_1$ 
in~$\bG^F$. Using the explicit expression (in terms of Chevalley 
generators of~$\bG^F$) for $x_{16}$ in \cite[Table~5]{Shof4}, we can
check that $s_1'$ commutes with~$x_{16}$. Hence, we have $g_1:=
s_1'x_{16}\in \Sigma^F$; since $g_1=x_{16}^2$, we must have $\zeta=-1$ 
for this choice of~$g_1$.
\end{abs}

\begin{abs} {\bf The cases where $p=2,3$.} \label{f4p23} In the above
discussion, we assumed that $p\neq 2,3$. For $p=2$, the scalars
$\zeta$ in the identities $R_x=\zeta \chi_A$ have been determined 
by Marcelo--Shinoda \cite[\S 4]{MaSh} and \cite[\S 5]{padua}. Now 
assume that $p=3$. For those cuspidal character sheaves $A$ where the
corresponding conjugacy class $\Sigma$ is unipotent (there are three 
of them), the scalars $\zeta$ are also determined by 
\cite[\S 4]{MaSh}. By \cite[\S 7.2]{S2}, the remaining four cuspidal 
character sheaves are analogous to those denoted above by $A_2$, $A_5$, 
$A_6$ and $A_7$. One checks that the discussions in 
\S \ref{csha56}, \ref{csha7}, \ref{csha2} can be applied 
almost verbatim to the case $p=3$, and yield the same results. The
Green functions for $p=3$ are known by \cite{MaSh} (see also
\cite[\S 5]{ekay})).
\end{abs}

\medskip
\noindent {\bf Acknowledgements}. I am indebted to Jean Michel for much
help with his programs \cite{jmich}, for independently verifying some of
my computations, and for pointing out the reference Bonnaf\'e \cite{Bo2}. 
Many thanks are due to Lacri Iancu and Gunter Malle for a careful 
reading of the manuscript, which led to the correction of several
inaccuracies and to a number of improvements of the exposition. This 
article is a contribution to Project-ID 286237555~–~TRR 195~--~by the 
Deutsche Forschungsgemeinschaft (DFG, German Research Foundation).

%%%%%%%%%%%%%%%%%%%%%%%%%%%%%%%%%%%%%%%%%%%%%%%%%%%%%%%%%%%%%%%%%%%%%%%%%%%

\bigskip
\noindent Meinolf Geck\\
    FB Mathematik\\ Universit\"at Stuttgart\\
    Pfaffenwaldring 57\\ 70569 Stuttgart, Germany\\
     Email: meinolf.geck@mathematik.uni-stuttgart.de\\
\end{document}